\PassOptionsToPackage{unicode}{hyperref}
\PassOptionsToPackage{naturalnames}{hyperref}
\documentclass[10pt,oneside,reqno]{amsart}
\usepackage[utf8]{inputenc}
\usepackage[english]{babel}
\usepackage{amsmath}
\numberwithin{equation}{section}
\usepackage{amssymb}
\usepackage{amsfonts}
\usepackage{amsthm}
\usepackage{mathrsfs}
\usepackage{mathtools}
\usepackage{geometry}
\usepackage{hyperref}
\theoremstyle{plain}
\newtheorem{Th}{Theorem}[section]
\newtheorem{Lem}[Th]{Lemma}
\newtheorem{proposition}[Th]{Proposition}

\theoremstyle{definition}
\newtheorem{Def}[Th]{Definition}

\theoremstyle{remark}
\newtheorem*{Remark}{Remark}

\begin{document}
\sloppy


\begin{center}

\textbf{Representations of simple noncommutative Jordan superalgebras I}
\footnote{This work was supported by FAPESP 16/16445-0.}

\bigskip

\textbf{Yury Popov}

\medskip

Universidade Estadual de Campinas, IMECC, Campinas, SP, Brazil
\sloppy
\sloppy

\medskip


\medskip

\section*{Abstract}
\end{center}

\medskip

In this article we begin the study of representations of simple finite-dimensional noncommutative Jordan superalgebras. In the case of degree $\geq 3$ we show that any finite-dimensional representation is completely reducible and, depending on the superalgebra, quasiassociative or Jordan. Then we study representations of superalgebras $D_t(\alpha,\beta,\gamma)$ and $K_3(\alpha, \beta, \gamma)$ and prove the Kronecker factorization theorem for superalgebras $D_t(\alpha,\beta,\gamma)$. In the last section we use a new approach to study noncommutative Jordan representations of simple Jordan superalgebras.

\medskip

\section{Introduction}
Noncommutative Jordan algebras were introduced by Albert in \cite{Alb}. He noted that the structure theories of alternative and Jordan algebras share so many nice properties that it is natural to conjecture that these algebras are members of a more general class with a similar theory. So he introduced the variety of noncommutative Jordan algebras defined by the Jordan identity and the flexibility identity. The class of noncommutative Jordan (super)algebras turned out to be vast: for example, apart from alternative and Jordan (super)algebras it contains quasiassociative (super)algebras, quadratic flexible (super)algebras and (super)anticommutative (super)algebras. However, the structure theory of this class is far from being nice.

Nevertheless, a certain progress was made in the study of structure theory of noncommutative Jordan algebras. Particularly, simple algebras of this class were studied by many authors. Schafer proved that a simple finite-dimensional noncommutative Jordan algebra over a field of characteristic 0 is either a Jordan algebra, or a quasiassociative algebra, or a flexible algebra of degree 2 \cite{Sch}. Oehmke proved an analog of Schafer's classification for simple flexible algebras with strictly associative powers and of characteristic $\neq 2, 3$ \cite{Oeh}, McCrimmon classified simple noncommutative Jordan algebras of degree $>2$ and characteristic $\neq 2$ \cite{McC,McC2}, and Smith described such algebras of degree 2 \cite{Smith}. The case of nodal (degree 1) simple algebras of positive characteristic was considered in the articles of Kokoris \cite{Kok,Kok2}, and the case of infinite-dimensional noncommutative Jordan superalgebras was studied by Shestakov and Skosyrskiy \cite{Shest, Skos}.

Simple finite-dimensional Jordan superalgebras over algebraically closed fields of characteristic 0 were classified by Kac \cite{Kac1} and Kantor \cite{Kan}. The study of superalgebras in positive characteristic was initiated by Kaplansky \cite{Kapl}. Racine and Zelmanov \cite{RZ} classified finite-dimensional Jordan superalgebras of characteristic $\neq 2$ with semisimple even part, the case where even part is not semisimple was considered by Mart\'{\i}nez and Zelmanov in \cite{MZ}, and Zelmanov considered the non-unital case in \cite{Zel}.Cantarini and Kac described  linearly compact simple Jordan and generalized Poisson superalgebras in \cite{CK}.

Simple finite-dimensional central noncommutative Jordan superalgebras were described by Pozhidaev and Shestakov in \cite{ps1,ps2,ps3}.
Structure and derivations of low-dimensional simple noncommutative Jordan algebras were studied in \cite{KLP17}. Generalizations of derivations of simple noncommutative Jordan superalgebras were studied in \cite{kay10}.

We briefly recall the history of the structure theory of other classes of algebras generalizing Jordan algebras. In the paper \cite{Kan_cons} Kantor generalized the Tits–Koecher–Kantor construction, extending it to the wide class of algebras, which he called conservative algebras. The Kantor construction puts a graded Lie algebra into correspondence with a conservative algebra. A conservative algebra is said to be of order 2 if its Lie algebra has $(-2,2)$-grading. In the same paper, he classified the finite dimensional simple conservative algebras of order 2 over an algebraically closed field of characteristic 0. 

In \cite{Allison}, Allison defined a class of nonassociative algebras containing the class of Jordan algebras and allowing the construction of generalizations of the structure algebra and the Tits–Koecher–Kantor algebra (Allison's construction). The algebras in this class, called structurable algebras, are unital algebras with involution. The class is defined by an identity of degree 4 and includes associative algebras, Jordan algebras (with the identity map as involution), tensor product of two composition algebras, the 56-dimensional Freudenthal module for $E_7$ with a natural binary product, and some algebras constructed from hermitian forms in a manner generalizing the usual construction of Jordan algebras from quadratic forms. Central simple finite dimensional structurable algebras over a field of characteristic zero were classified by Allison. The classification of simple structurable algebras over a modular field was obtained by Smirnov \cite{Smirnov}. Moreover, he found a new class of simple structurable algebras of characteristic zero, missed by Allison. Simple structurable superalgebras over an algebraically closed field of characteristic 0 were described by Faulkner \cite{Faulkner} and by Pozhidaev and Shestakov \cite{ps_str}.

Representations of alternative and Jordan superalgebras are considered in various works. In the paper \cite{MZ} Mart\'{i}nez and Zelmanov used the Tits-Koecher-Kantor construction to describe superbimodules over superalgebras $JP(n), M_{m,n}(\mathbb{F})^{(+)}, Josp(m,2r)$ and Jordan superalgebras of supersymmetric bilinear superforms over algebraically closed fields of characteristic 0. Also they proved that the universal enveloping superalgebra for unital representations of simple finite-dimensional Jordan superalgebra of degree $\geq 3$ is finite-dimensional and semisimple and that every representation over such superalgebra is completely reducible. Some of the Mart\'{\i}nez-Zelmanov results were generalized to the case of arbitrary characteristic $\neq 2.$ For example, representations of superalgebras $JP(n)$ and $Q(n)^{(+)}, n \geq 2$ over fields of characteristic $\neq 2$ were considered by Mart\'{\i}nez, Shestakov and Zelmanov in \cite{MSZ}. Irreducible representations of superalgebras of Poisson-Grassmann bracket were classified by Shestakov and Solarte in \cite{SS}. In the papers \cite{Trush}, \cite{Trushmod} Trushina described irreducible bimodules over the superalgebras $D_t$ and $K_3$.  In the work \cite{MZ01} the universal envelopes for one-sided representations of simple Jordan superalgebras were constructed, and also irreducible one-sided bimodules over the superalgebras $D_t$ were described. Shtern \cite{Sht} classified irreducible bimodules over the exceptional Kac superalgebra $K_{10}$.

Representations of alternative superalgebras were studied by Pisarenko. Particularly, he proved the following: let $A$ be a finite-dimensional semisimple alternative superalgebra over a field of characteristic $\neq 2, 3$. If $A$ contains no ideals isomorphic to the two-dimensional simple associative superalgebra, then every bimodule over $A$ is associative and completely reducible \cite{Pis}. For the case of two-dimensional simple associative superalgebra he obtained a series of indecomposable alternative superbimodules. L\'{o}pez-D\'{\i}az and Shestakov described irreducible superbimodules and proved the analogues of Kronecker factorization theorem for exceptional alternative and Jordan superalgebras of characteristic 3 in \cite{LDSA,LDSJ}. Infinite-dimensional representations of alternative superalgebras were studied in the paper \cite{ST}.

In the present article we begin the study of representations of central simple finite-dimensional noncommutative Jordan superalgebras.

This paper is organized as follows. In section 2 we provide all the preliminary information which will be necessary to work with noncommutative Jordan superalgebras and their representations. Then we reformulate the definitions of a noncommutative Jordan (super)algebra and representation in terms of Jordan multiplication and Poisson brackets. In section 3 we classify finite-dimensional representations over simple noncommutative Jordan superalgebras of degree $\geq 3.$ In section 4 we describe superbimodules over superalgebras $D_t(\alpha,\beta,\gamma)$ and $K_3(\alpha, \beta, \gamma)$. In section 5 we prove the Kronecker factorization theorem for superalgebras $D_t(\alpha,\beta,\gamma)$ and use it to study representations of the superalgebra $Q(2)$ in the section 6. In the last section we classify noncommutative Jordan representations over some simple Jordan superalgebras.

\medskip

\section{Preliminaries}
In this section we briefly recall the definitions, techniques and objects which we work with throughout the paper. Also here we reproduce the classification of central simple finite-dimensional noncommutative Jordan superalgebras by Pozhidaev and Shestakov.

\subsection{Notations and defining identities}
Since this paper deals with representations of nonassociative superalgebras, ''(super)algebra'' means a not necessarily associative (super)algebra, and ''module'' and ''representation'' mean respectively a (super)bimodule and a two-sided representation over a (super)algebra, if not explicitly said otherwise. Also, occasionally we drop the prefix ''super-'', but it should be always clear from the context in which setting we are working. All algebras and vector spaces in this paper are over a field $\mathbb{F}$ of characteristic not 2. As this work follows up on and uses the formulas from the papers \cite{ps1} and \cite{ps2}, the operators in this paper act on the right.

\medskip

For a subset $S$ of a $\mathbb{F}$-vector (super)space by $\langle S \rangle$ we denote its $\mathbb{F}$-span.

\medskip

Let $U = U_{\bar{0}} + U_{\bar{1}}$ be a superalgebra. In what follows, if the parity of an element arises in a formula, this element is assumed to be homogeneous. Idempotents are also assumed to be homogeneous. We assume the following standard notation:
\[(-1)^{xy}= (-1)^{p(x)p(y)},\]
where $p(a)= i,$ if $a \in U_{\bar{i}}$ is the parity of $a$, and
\[(-1)^{x,y,z}= (-1)^{xy + xz + yz}.\]

\begin{Def} Let $A$ and $B$ be superalgebras. By $A \otimes B = C$ we denote their \textit{graded tensor product}, which is defined as
\[C_{\bar{0}} = A_{\bar{0}}\otimes B_{\bar{0}} + A_{\bar{1}} \otimes B_{\bar{1}},~ C_{\bar{1}} = A_{\bar{0}}\otimes B_{\bar{1}} + A_{\bar{1}} \otimes B_{\bar{0}},\]
and the multiplication is given by
\[(a \otimes b)\cdot (a' \otimes b') = (-1)^{a'b}(aa')\otimes (bb').\]
\end{Def}

By $L_x, R_x$ we denote the operators of left and right multiplication by $x \in U: $
\[L_x:= (-1)^{xy}xy,~ yR_x:= yx.\]

The supercommutator and super Jordan product are also denoted in the standard manner:
\[[x,y]:= xy - (-1)^{xy}yx,~ x \circ y:= (xy + (-1)^{xy}yx)/2,~ x \bullet y:= xy + (-1)^{xy}yx.\]

\begin{Def}
The (super)algebra $(U,\circ)$ is called \textit{the symmetrized (super)algbera of $U$} and is denoted by $U^{(+)}.$
\end{Def}

\begin{Def} A supercommutative superalgebra $J$ is called \emph{Jordan} if it satisfies the following operator identity:
\begin{equation}
\label{jord_identity}
R_aR_bR_c + (-1)^{a,b,c}R_cR_bR_a + (-1)^{bc}R_{(ac)b} = R_aR_{bc} + (-1)^{a,b,c}R_cR_{ba} + (-1)^{ab}R_bR_{ac}.
\end{equation}
\end{Def}

\begin{Def} A superalgebra $U$ is called \emph{noncommutative Jordan} if it satisfies the following operator identities:
\begin{equation}
\label{noncomm_jord_identity}
[R_{a \circ b}, L_c] + (-1)^{a(b+c)}[R_{b\circ c}, L_a] + (-1)^{c(a+b)}[R_{c \circ a}, L_b] = 0,
\end{equation}
\begin{equation}
\label{flex}
[R_a, L_b] = [L_a, R_b].
\end{equation}
\end{Def}
The identity (\ref{flex}) defines the class of \emph{flexible superalgebras.}
If we assume that all elements of $U$ are even we arrive at the notion of a noncommutative Jordan algebra.

The flexibility identity may be written as
\begin{equation}
\label{flex_1}
(-1)^{ab}L_{ab} - L_bL_a = R_{ba} - R_bR_a,
\end{equation}
or
\begin{equation}
\label{flex_2}
(x,y,z) = (-1)^{x,y,z}(z,y,x).
\end{equation}

\medskip

We would like to clarify the origin of the name ''noncommutative Jordan''. One can check that a Jordan (super)algebra is noncommutative Jordan. On the other hand, a \textit{commutative} noncommutative Jordan (super)algebra is Jordan --- hence the name. In fact, the relation between noncommutative and commutative Jordan (super)algebras can be made precise using the notion of the symmetrized algebra:
\begin{Lem}[\cite{ps1}]
\label{flexplus}
$U$ is a noncommutative Jordan (super)algebra if and only if $U$ is a flexible (super)algebra such that its symmetrized (super)algebra $U^{(+)}$ is a Jordan (super)algebra.
\end{Lem}

Using this lemma it is easy to see that (super)algebras from many well-known varieties, such as associative, alternative and anticommutative (super)algebras, are noncommutative Jordan. So while reading this paper it is useful to bear in mind associative (or alternative) and Jordan (super)algebras as examples.

\medskip

From now on, unless said otherwise, we denote by $U$ a noncommutative Jordan superalgebra over a field $\mathbb{F}$, and by $e$ an even idempotent in $U$.

\medskip

\subsection{Peirce decomposition}  Here we recall some usual facts about Peirce decomposition, which is going to be our main tool during the paper. A more detailed exposition of Peirce decomposition can be found in \cite{McC}, \cite{ps1}.

 The identity (\ref{noncomm_jord_identity}) can be shown to be equivalent to the identity
\begin{equation}
\label{noncomm_jord_identity_1}
R_{a(b\bullet c)} - R_a R_{b\bullet c} + (-1)^{ab}(R_b + L_b)(L_aL_c - (-1)^{ca}L_{ca}) + (-1)^{c(a+b)}(R_c + L_c)(L_aL_b - (-1)^{ab}L_{ba}) = 0.
\end{equation}
Substituting $a = b = c = e$ in (\ref{noncomm_jord_identity_1}) gives us
\[R_e + (R_e + L_e)L_e^2 = (R_e + L_e)L_e + R_e^2.\]
By (\ref{flex_1}) we have $L_e - L_e^2 = R_e - R_e^2$. Hence, the last equation is equivalent to
$(R_e + L_e)(L_e - L_e^2) = (L_e - L_e^2).$ Put
\[U_i = \{x \in U |ex + xe = ix\} \mbox{ for } i = 0,1,2.\]
Using the standard argument we get the decomposition
\begin{equation}
\label{pd}
U = U_0 \oplus U_1 \oplus U_2.
\end{equation}
\begin{Def}
The decomposition (\ref{pd}) is called the \textit{Peirce decomposition of $U$ with respect to $e,$} and the spaces $U_i = U_i(e)$ are called \textit{Peirce spaces.}
\end{Def}
The identities (\ref{noncomm_jord_identity}) and (\ref{flex}) imply that
\begin{equation}
\label{comm_op}
[E_x,F_e] = 0,
\end{equation}
 if $x \in U_0 + U_2, \{E,F\} \subseteq \{R,L\}.$ Denote by $P_i$ the associated projections on $U_i$ along the direct sum of two other Peirce spaces. Since $P_i$ are polynomials in $L_e + R_e,$
\begin{equation}
\label{comm_proj}
[E_x,P_i] = 0,
\end{equation}
if $x \in U_0 + U_2, E \in \{R,L\}$. The spaces $U_0, U_1$ and $U_2$ satisfy the following relations (which we call \textit{the Peirce relations}):
\begin{equation}
\label{pd_1}
U_i^2 \subseteq U_i, U_iU_1 + U_1U_i \subseteq U_1, i = 0, 2; ~ U_0U_2 = U_2U_0 = 0,
\end{equation}
\begin{equation}
\label{pd_2}
x \in U_i, i = 0, 2 \Rightarrow xe = ex = \frac{1}{2}ix; ~ x, y \in U_1 \Rightarrow x \circ y \in U_0 + U_2.
\end{equation}

These relations will be used so frequently throughout the paper that referencing them each time would make it messy. Therefore, we only occasionally explicitly reference them (for example, when it is not clear which relation we use), and in other cases when we apply them we say ''by Peirce relations'', or use them without mentioning. 

\medskip

If $e = \sum_{i=1}^n e_i$ is a sum of orthogonal idempotents, then analogously one can obtain the \textit{Peirce decomposition with respect to $e_1,\ldots,e_n$:}
\begin{equation}
\label{pd_n}
U = \bigoplus_{i,j=0}^n U_{ij},
\end{equation}
where
\begin{align*}
U_{00} &= \{x \in U | e_ix = xe_i = 0 \text{ for all } i\},\\
U_{ii} &= \{x \in U | e_ix = xe_i = x, e_jx = xe_j = 0, j \neq i\},\\
U_{i0} &= \{x \in U | e_ix + xe_i = x, e_jx + xe_j = 0, j \neq i\} = U_{0i},\\
U_{ij} &= \{x \in U | e_ix + xe_j = e_jx + xe_j = x, e_kx + xe_k = 0, k \neq i, j\} = U_{ji}.
\end{align*}

Note that if $i, j \neq 0$ and $x \in U_{ij},$ then $e_ix = xe_j.$

As above, there are associated projections $P_{ij}$ on $U_{ij}$ and the following inclusions hold:
\begin{gather*}
U_{ii}^2 \subseteq U_{ii}, U_{ii}U_{ij} + U_{ij}U_{ii} \subseteq U_{ij},\\
U_{ij}U_{jk} + U_{jk}U_{ij} \subseteq U_{ik}, U_{ij}^2 \subseteq U_{ii} + U_{ij} + U_{jj}.
\end{gather*}
for distinct $i, j, k$ (all other products are zero).

Clearly, the decompositions above apply to any subspace $M \subseteq U$ invariant under the multiplication by $e$, for example, to an ideal of $U$.

\medskip

With the aid of the following lemmas one can restore some of the original products in $U$ from the products in $U^{(+)}$ and multiplication by $e.$

\begin{Lem}[\cite{ps1}] If $x, y \in U_0, u \in U_2, u_i \in U_i (i = 0, 2) \text{ and } z, w \in U_1,$ then
\begin{equation}
\label{u0_mult}
e(z\bullet y) = ez \bullet y =  zy,~ (y \bullet z) e = y \bullet ze = yz;
\end{equation}
\begin{equation}
\label{u2_mult}
e(u\bullet z) = u \bullet ez = uz,~ (z \bullet u) e = ze \bullet u = zu;
\end{equation}
\begin{equation}
\label{p_i}
P_2(ez\bullet w) = P_2(z \bullet we) = P_2(zw),~ P_0(w \bullet ez) = P_0(we \bullet z) = P_0(wz);
\end{equation}
\begin{equation}
\label{p_1}
P_1(zw)\bullet u_i =  P_1(z(w\bullet u_i)) = (-1)^{wu_i}P_1((z\bullet u_i)w).
\end{equation}
\end{Lem}

\begin{Lem} For $x, y, z \in U_1$ the following relation holds:
\begin{equation}
\label{p_ip_1}
x\circ P_1(yz) = P_1(xy)\circ z = (-1)^{x(y+z)}y \circ P_1(zx).
\end{equation}
\end{Lem}
\begin{proof}
From \cite{ps1}, Lemma 4.1, it follows that for $i \in \{0,2\}$ and $x, y, z \in U_1$ the following relation holds:
\[P_i(x\circ P_1(yz)) = P_i(P_1(xy)\circ z) = (-1)^{x(y+z)}P_i(y \circ P_1(zx)).\]
Since $U_1 \circ U_1 \subseteq U_0 + U_2$ by (\ref{pd_2}), summing the relations for $i =0$ and $i =2$ yields the desired relation.
\end{proof}

This is an obvious yet useful consequence of the last lemma:
\begin{Lem} \textit{ Let $x, y \in U_1$ be such that $xy \in U_0 + U_2.$ Then}
\begin{equation}
\label{l_xp_1r_y+l_y}
P_1L_xP_1(R_y + L_y) = 0,~ P_1R_xP_1(R_y + L_y) = 0.
\end{equation}
\end{Lem}

Note that in contrast to associative or Jordan superalgebras, for arbitrary noncommutative Jordan superalgebra the inclusion $U_1^2 \subseteq U_0 + U_2$ does not hold (see relation (\ref{pd_2})). However, one can find a sufficient condition to ensure that this inclusion holds:
\begin{Lem}
\label{K_set}
Suppose that there exists a subset $K \subseteq U_1$ such that\\
$1)$ $KU_1 \subseteq U_0 + U_2;$\\
$2)$ $a \in U_1, K \circ a = 0 \Rightarrow a = 0.$\\
Then $U_1^2 \subseteq U_0 + U_2.$
\end{Lem}
\begin{proof} Indeed, the relation (\ref{p_ip_1}) implies that for $a, b \in U_1$ we have $K\circ P_1(ab) = P_1(Ka) \circ b = 0$ by the first condition of the lemma. Hence, the second condition of the lemma implies that $P_1(ab) = 0.$
\end{proof}

We will also need a technical lemma by Pozhidaev and Shestakov:
\begin{Lem}[\cite{ps1}]
\label{psu1id}
For $a, b \in U_i,\ i=0,2$, the following operator identities hold in $U_1:$
\begin{gather*}
R_{ab} = R_aR_b+(-1)^{ab}L_bR_a = R_aR_b+(-1)^{ab}R_bL_a;\\
L_{ab} = (-1)^{ab}L_bL_a + L_aR_b = (-1)^{ab}L_bL_a + R_aL_b.
\end{gather*}
\end{Lem}

\medskip

\subsection{Algebras with connected idempotents}
For some well-studied classes of (super)algebras, such as associative and Jordan, the Peirce relations are in fact stronger than (\ref{pd_1}), (\ref{pd_2}), which is a consequence of their defining identities. For example, if $J$ is a Jordan (super)algebra, then its Peirce space $J_1$ is obviously
\[J_1 = \{x \in J: xe = x/2\},\]
and if $A$ is an associative (super)algebra, then its Peirce space $A_1$ decomposes as
\begin{gather*}
A_1 = A_{10}\oplus A_{01}, \mbox{ where}\\
A_{10}= \{x \in A: ex = x, xe = 0\},\\
A_{01} = \{x \in A: ex = 0, xe = x\}.
\end{gather*}
So, in nice cases the space $U_1$ decomposes in the direct sum of eigenspaces of $L_e.$ Moreover, if $U$ is associative or Jordan, then $U_1^2 \subseteq U_0 + U_2.$

In this subsection we consider the general situation, introducing $L_e$-eigenspaces in $U_1,$ and showing that they satisfy properties analogous to ones described above. Then we introduce related notions of connectedness of idempotents and the degree of an algebra. Finally, we state the results that show that if an algebra $U$ is ''sufficiently large'' (i.e., has degree $\geq 3$) and if ''eigenspaces'' of $U_1$ satisfy certain conditions, then $U$ is associative or Jordan.

\medskip

For $\lambda \in \mathbb{F}$ consider the space $U_1^{[\lambda]} = \{x \in U_1 | L_ex = \lambda x\}.$ This set is invariant under the multiplication by $U_i,i = 0,2:$
\begin{Lem}[\cite{ps1}]
\label{U_lambda}
\textit{ $U_iU_1^{[\lambda]} + U_1^{[\lambda]}U_i \subseteq U_1^{[\lambda]}$ for $i = 0, 2.$}
\end{Lem}

The following technical lemmas will simplify further computations:
\begin{Lem}
\label{lambda}
Let $x \in U_1^{[\lambda]}.$ Then$:$\\
$1)$ $(1 - \lambda)((id-L_e)L_x + R_eR_x) - \lambda((id-R_e)R_x + L_eL_x) = 0;$\\
$2)$ $\lambda(L_x(id-L_e) + R_xR_e) - (1 - \lambda)(L_xL_e + R_x(1-R_e)) = 0.$
\end{Lem}
\begin{proof} 1) Let $a = x, b = e$ in (\ref{flex_1}):
\[0 = (1 - \lambda)L_x - (1 - \lambda + \lambda)L_eL_x - \lambda R_x + (1 - \lambda + \lambda)R_eR_x = \]
\[(1 - \lambda)(L_x - L_eL_x + R_eR_x) - \lambda(R_x + L_eL_x - R_eR_x),\]
which proves the first point of the lemma.

2) Let $a = e, b = x$ in (\ref{flex_1}):
\[0 = \lambda L_x - (1 - \lambda + \lambda)L_xL_e - (1 - \lambda)R_x + (1 - \lambda + \lambda)R_xR_e = \]
\[\lambda(L_x - L_xL_e + R_xR_e) - (1-\lambda)(L_xL_e + R_x - R_xR_e),\]
which proves the second point of the lemma.
\end{proof}

\begin{Lem}
\label{U1^iU_i}
$1)$ $U_1^{[0]}U_0 = U_2U_1^{[0]} = 0;$\\
$2)$ $U_0U_1^{[1]} = U_1^{[1]}U_2 = 0.$
\end{Lem}
\begin{proof}
Let $a \in U_0, b = e$ in (\ref{flex_1}): $L_eL_a = -R_eR_a.$ Applying this relation on $z \in U_1^{[0]},$ we get $za = 0.$ Now let $a \in U_2, b = e$ in (\ref{flex_1}): $L_a - L_eL_a = R_a - R_eR_a.$ Applying this relation on $z \in U_1^{[0]}$, we get $az = 0.$ The second point of the lemma is proved completely analogously.
\end{proof}

\begin{Lem}
\label{U1^iU_1^i}
$1)$ $U_1^{[0]}U_1 \subseteq U_0 + U_1,~ U_1U_1^{[0]} \subseteq U_1 + U_2;$\\
$2)$ $U_1^{[1]}U_1 \subseteq U_1 + U_2,~ U_1U_1^{[1]} \subseteq U_0 + U_1;$\\
$3)$ $(U_1^{[0]})^2 \subseteq U_1,~ (U_1^{[1]})^2 \subseteq U_1.$
\end{Lem}
\begin{proof}
First two points of the lemma follow from (\ref{p_i}), and the third point follows from the first two.
\end{proof}

For any $\lambda \in \mathbb{F}$ the space $S_1^{[\phi]}(e) = U_1^{[\lambda]} + U_1^{[1-\lambda]}$ is completely determined by the value $\phi = \lambda(1 - \lambda)$ and can be thought of as the ''eigenspace'' of $L_e$ and $R_e$ corresponding to the ''eigenvalue'' $\phi$. Let (\ref{pd_n}) be the Peirce decomposition of $U$ with respect to a system of orthogonal idempotents $e_1,\ldots,e_n.$ Put
$$S_{ij}^{[\phi]} = S_1^{[\phi]}(e_i) \cap S_1^{[\phi]}(e_j).$$

\begin{Def} We say that $e_i$ and $e_j$ are \textit{evenly connected} if there is a scalar $\phi \in \mathbb{F}$ and even elements $v_{ij}, u_{ij} \in S^{[\phi]}_{ij}$
such that $v_{ij}u_{ij} = u_{ij}v_{ij} = e_i + e_j , i < j$. We say that $e_i$ and $e_j$ are \textit{oddly connected} if there is a scalar $\phi \in \mathbb{F}$ and
odd elements $v_{ij}, u_{ij} \in S^{[\phi]}_{ij}, i < j,$ such that $v_{ij}u_{ij} = -u_{ij}v_{ij} = e_i - e_j$. Lastly, $e_i$ and $e_j$ are said to be \textit{connected} if they are either evenly or oddly connected. The element $\phi$ is called \emph{an indicator} of $U_{ij}$.
\end{Def}

\begin{Def} We say that a noncommutative Jordan superalgebra $U$ \textit{is of degree $k$} if $k$ is the maximal possible number of pairwise orthogonal connected idempotents in $U \otimes_{\mathbb{F}} \overline{\mathbb{F}}$ , where
$\overline{\mathbb{F}}$ is the algebraical closure of $\mathbb{F}$. And say that $U$ \textit{has unity of degree $k$} if $k$ is the degree of $U$ and the unity of $U$ is a sum of $k$ orthogonal pairwise connected idempotents.
\end{Def}

A classical situation in theory of Jordan algebras is that a lot can be said about the structure and representations of a Jordan (super)algebra if it has degree $\geq 3$ (see, for example, the coordination theorem in \cite{Jac}). This stays true for the noncommutative case as well: McCrimmon partially described the structure of noncommutative Jordan algebras with unity of degree $\geq 3,$ and Pozhidaev and Shestakov generalized his results for superalgebras. We state their results:
\begin{Lem}[\cite{ps1}]
\label{ind_type}
If $U$ has unity of degree $\geq 3$, then all indicators have the common value $\phi$ $($i.e., $U_{ij} = S_{ij}^{[\phi]}, i \neq j)$. The element $\phi \in \mathbb{F}$ is then called the \textit{indicator} of $U,$ and $U$ is said to be of indicator type $\phi$.
\end{Lem}
\begin{Lem}[\cite{ps1}]
\label{ind1/4}
If $U$ has unity of degree at least 3 and is of indicator type $\phi = 1/4$, then $U$ is supercommutative.
\end{Lem}
\begin{Lem}[\cite{ps1}]
\label{ind0}
If $U$ has unity of degree at least 3 and is of indicator type $\phi = 0$, then $U$ is associative.
\end{Lem}

It seems that the results stated above only work for specific values of an $L_e$-eigenvalue $\lambda$ and indicator $\phi.$ In fact, we can control the indicator type of $U$ and the eigenvalue $\lambda$ using the construction called mutation, which we treat in the next subsection.
\medskip

\subsection{Mutations} Mutation is a construction which generalizes the symmetrization of an algebra: $A \to A^{(+)}$. Since the class of noncommutative Jordan algebras is large, it is closed under mutations. In fact, the process of mutation is almost always invertible, so it does not really give new interesting examples of algebras. However, using mutations we may greatly simplify the multiplication table of an algebra, and also they allow us to formulate our results in a concise way, so they are still useful.

\medskip

\begin{Def}Let $A = (A, \cdot)$ be a superalgebra over $\mathbb{F}$ and $\lambda \in \mathbb{F}.$ By $A^{(\lambda)}$ we denote the superalgebra $(A, \cdot_\lambda),$ where
\[x\cdot_\lambda y = \lambda x \cdot y + (-1)^{xy} (1-\lambda)y \cdot x.\]
The superalgebra $A^{(\lambda)}$ is called the \emph{$\lambda$-mutation} of $U$.
\end{Def}
It is easy to see that $A^{(1/2)}$ is the symmetrized superalgebra $A^{(+)},$ and $A^{(0)}$ is the opposite superalgebra $A^{op}.$

\medskip

Since $L_x^{\lambda} = \lambda L_x + (1-\lambda)R_x, R_x^{\lambda} = \lambda R_x + (1- \lambda)L_x,$ it is easy to see (by plugging the new operators in relations (\ref{noncomm_jord_identity}), (\ref{flex})) that a mutation of a noncommutative Jordan superalgebra is again a noncommutative Jordan superalgebra.
For example, a mutation $A^{(\lambda)}$ of an associative superalgebra $A$ is called a \emph{split quasiassociative superalgebra}. A superalgebra $U$ is called a \emph{quasiasscociative superalgebra} if there exists an extension $\Omega$ of $\mathbb{F}$ such that $U_{\Omega} = U \otimes_{\mathbb{F}} \Omega$ is a split quasiassociative superalgebra over $\Omega.$

\medskip

Consider a double mutation $(A^{(\lambda)})^{(\mu)}.$ One can compute that $(A^{(\lambda)})^{(\mu)} = A^{(\lambda \odot \mu)},$ where $\lambda \odot \mu = 2\lambda\mu - \lambda - \mu + 1.$ Hence, if $\lambda \neq 1/2,$ there exists $\mu \in \mathbb{F}$ such that $\lambda \odot \mu = 1,$ and we can recover $A$ from $A^{(\lambda)}: A = A^{(1)} = A^{(\lambda \odot \mu)} = (A^{(\lambda)})^{(\mu)}.$ However, if $\lambda = 1/2,$ it is impossible to immediately recover $A$ from $A^{(1/2)} = A^{(+)},$ since, for example, all mutations of $A$ have the same $A^{(+)}.$

\medskip

Many results about noncommutative Jordan algebras can be formulated using mutations. For example, in case of degree $\geq 3$ we have the noncommutative coordinatization theorem:

\begin{Th}[Coordinatization theorem, \cite{ps1}]
\label{coord_th}
Let $\mathbb{F}$ be a field which allows square root extraction and $U$ be a noncommutative Jordan superalgebra with unity of degree $n \geq 3$ which is not supercommutative. Then $U = (A_n)^{(\lambda)}$ is the $\lambda$-mutation of the $n\times n$ matrix algebra over an associative superalgebra $A$ for $\lambda \in \mathbb{F}.$
\end{Th}
\begin{proof}
We only give the main idea of the proof. If $U$ is not supercommutative, then its indicator type $\phi = \lambda(1-\lambda) \neq 1/4$ by lemma \ref{ind1/4}. Therefore, $\mathbb{F} \ni \lambda \neq 1/2$ and there exists $\mu \in \mathbb{F}$ such that $\lambda \odot \mu = 1.$ Now, one can check that $U^{(\mu)}$ is of indicator type 0. Thus, by Lemma \ref{ind0} it is associative and is in fact the $n\times n$ matrix superalgebra over the superalgebra $A = U_{11}.$ Hence, by the double mutation rule, $U = (A_n)^{(\lambda)}.$
\end{proof}

Now we can study simple algebras. Note that to have a reasonable structure theory, we have to exclude the anticommutative (super)algebras somehow. This is done simply by declaring that \textit{a simple (super)algebra does not contain nontrivial nil ideals} (i.e., choosing the nil radical as the radical).

\begin{Remark}
In the paper \cite{Shest} it is shown that a simple nil noncommutative Jordan algebra is anticommutative.
\end{Remark}

It is easy to see that an ideal in $A$ remains an ideal in $A^{(\lambda)}.$ Hence, if $\lambda \neq 1/2$, ideals in $A$ and $A^{(\lambda)}$ coincide. In particular, if $U$ is a simple noncommutative Jordan superalgebra, then all its $\lambda$-mutations, $\lambda \neq 1/2$, are simple noncommutative Jordan superalgebras.

The list of central simple noncommutative Jordan superalgebras clearly includes all central simple Jordan superalgebras. Also we understood that it also includes all simple quasiassociative superalgebras. Pozhidaev and Shestakov proved that in the case of finite dimension and degree $\geq 3$ there is nothing else:
\begin{Th} [\cite{ps1,ps3}]
\label{class}
A finite-dimensional central simple noncommutative Jordan superalgebra $U$ is either

$(1)$ of degree 1;

$(2)$ of degree 2;

$(3)$ quasiassociative;

$(4)$ supercommutative.

Also, if $U$ is of degree $n \geq 3,$ it has unity of degree $n.$
\end{Th}

Therefore, essentially new examples of simple noncommutative Jordan superalgebras must have degree $\leq 2.$ The next subsection gives an approach for their classification.

\medskip

\subsection{Poisson brackets} Poisson brackets, Poisson algebras and generic Poisson algebras are important objects in nonassociative algebra. Here we shall see that one can in fact give a definition of a noncommutative Jordan (super)algebra using the notion of generic Poisson bracket and its symmetrized (super)algebra. Also we see that in this context that the simplicity of a noncommutative Jordan (super)algebra is equivalent to the simplicity of its symmetrized (super)algebra in the case of degree $\geq 2.$

\begin{Def} A superanticommutative binary linear operation $\{\cdot,\cdot\}$ on a superalgebra $(A,\cdot)$ is called a \emph{generic Poisson bracket} \cite{ksu} if for arbitrary $a, b, c \in A$ we have
\begin{equation}
\label{pois_br}
\{a \cdot b, c\} = (-1)^{bc}\{a,c\}\cdot b + a \cdot \{b,c\}.
\end{equation}
\end{Def}
In other words, for any homogeneous $c \in A$ the map $\{\cdot,c\}$ is a derivation of degree $p(c).$

\medskip

Generic Poisson brackets are important in the study of noncommutative Jordan (super)algebras. We have already seen that a symmetrized (super)algebra of a noncommutative Jordan (super)algebra is a Jordan (super)algebra. So we may ask ourselves: how do we reproduce the original structure of a noncommutative Jordan (super)algebra $U$ having only its symmetrized (super)algebra $U^{(+)}?$ Or equivalently, which noncommutative Jordan (super)algebras have a symmetrized (super)algebra isomorphic to a given Jordan (super)algebra $J$? The answer can be formulated nicely in terms of (noncommutative) Poisson brackets:

\begin{Lem} [\cite{ps2}]
\label{JordPois}
Let $(J, \circ)$ be a Jordan superalgebra and $\{\cdot,\cdot\}$ be a generic Poisson bracket on $J.$ Then $(J,\cdot),$ where $a\cdot b = \frac{1}{2}(a\circ b + \{a,b\})$ is a noncommutative Jordan superalgebra. Conversely, if $U$ is a noncommutative Jordan superalgebra, then the supercommutator $[\cdot,\cdot]$ is a generic Poisson bracket on a Jordan superalgebra $U^{(+)}.$ Moreover, the multiplication in $U$ can be recovered by the Jordan multiplication in $U^{(+)}$ and the Poisson bracket: $ ab = \frac{1}{2}(a \circ b + [a,b]).$
\end{Lem}

Hence, we can \textit{define} a noncommutative Jordan superalgebra $U$ as a superalgebra with two multiplications:
\begin{Def}
\label{noncomm_jord_2mult}
A (super)algebra $U = U(\circ, \{\cdot,\cdot\})$ with two multiplications $\circ$ and $\{\cdot,\cdot\},$ which we call respectively \textit{circle} and \textit{bracket multiplications}, is called \textit{noncommutative Jordan}, if $(U,\circ) = J$ is a Jordan (super)algebra, and $\{\cdot,\cdot\}$ is a generic Poisson bracket on $J.$
\end{Def}

In this setting, an ideal of $(U,\circ,\{\cdot,\cdot\})$ is a subspace invariant with respect to both multiplications. Consider two marginal cases: if $\{\cdot,\cdot\} = 0,$ then $U$ is Jordan, if $\circ = 0,$ then $U$ is anticommutative. The passage to the symmetrized algebra in this setting is just forgetting the bracket multiplication: $U^{(+)} = (U,\circ).$

\medskip

Now, if $A$ is any algebra, it is obvious that if $A^{(+)}$ is simple, then $A$ is also simple. Pozhidaev and Shestakov proved that the converse holds in the case of noncommutative Jordan superalgebras of degree $\geq 2:$
\begin{Th}
Let $U$ be a central simple finite-dimensional noncommutative Jordan superalgebra of degree $\geq 2.$ Then $U^{(+)}$ is a simple finite-dimensional Jordan superalgebra.
\end{Th}


In other words, a simple simple finite-dimensional central noncommutative Jordan superalgebra of degree $\geq 2$ is a simple finite-dimensional Jordan superalgebra with a generic Poisson bracket. Therefore, to classify simple superalgebras in degree 2, it suffices to find all generic Poisson brackets on simple finite-dimensional Jordan superalgebras (which are known), and classify the resulting noncommutative superalgebras up to isomorphism. In the next subsection we consider some examples of this approach.

\medskip

So the structure theory of noncommutative Jordan (super)algebras can be formulated in a nice manner if we think of them as (super)algebras with two multiplications. In the last subsection of this section we will see what does the definition of a noncommutative Jordan representation look like in this context.

\subsection{Examples and classification in degree \texorpdfstring{$\leq 2$}{≤ 2}}
Here we provide some examples of noncommutative Jordan superalgebras of degree $\leq 2$ given in \cite{ps1,ps2}. We also correct a mistake in the classification of the algebras $D_t(\alpha, \beta, \gamma)$ from the paper \cite{ps1} (the algebra $D_t(1/2,1/2,0)$ was omitted there). In the end of the subsection we state the classification theorem for simple superalgebras of degree $\leq 2.$

\subsubsection{The superalgebra $D_t(\alpha, \beta, \gamma)$}
Let $t \in \mathbb{F}.$ Recall that a simple Jordan superalgebra $D_t$ is defined in the following way:
\begin{gather*}
D_t = (D_t)_{\bar{0}} \oplus (D_t)_{\bar{1}},~ (D_t)_{\bar{0}} = \langle e_1, e_2 \rangle,~ (D_t)_{\bar{1}} = \langle x, y \rangle,\\
e_i^2 = e_i,~ e_1\circ e_2 = 0, \\
e_1\circ x = e_2 \circ x = x/2,~ e_1 \circ y = e_2 \circ y = y/2,\\
x \circ y = -y \circ x = e_1 + te_2.
\end{gather*}

Suppose that $U$ is a noncommutative Jordan superalgebra such that $U^{(+)} = D_t.$ Here we describe such algebras and classify them up to isomorphism.

Let $U_0 = \langle e_2 \rangle, U_1 = \langle x, y \rangle, U_2 = \langle e_1 \rangle$ be the Peirce decomposition of $U$ with respect to $e_1.$ The Peirce relation (\ref{pd_1}) implies that
$$e_1x = \alpha x + \beta y, ~ e_1y = \gamma x + \delta y, ~ \alpha, \beta, \gamma, \delta \in \mathbb{F}.$$
Relation (\ref{p_i}) implies that $P_2(xy) = P_2(ex \bullet y) = \alpha P_2(x \bullet y) = 2\alpha e_1.$ Analogously, we obtain $P_2(yx) = -2\delta e_1.$ Hence, $\alpha + \delta = 1.$ Since $U_1 = U_{\bar{1}}$,
\[P_1(x^2) = 0,~ P_1(y^2) = 0,~ P_1(xy) = 0,~ P_1(yx) = 0.\]
Again using the relation (\ref{p_i}) we get
\begin{gather*}
P_2(x^2) = -2\beta e_1,~ P_0(x^2) = 2\beta t e_2,\\
P_2(y^2) = 2\gamma e_1,~ P_0(y^2) = -2\gamma t e_2,\\
P_0(xy) = 2(1-\alpha)te_2,~ P_0(yx) = -2\alpha t e_2.
\end{gather*}
Therefore, multiplication in $U$ is of the following form:
\begin{gather*}
e_i^2 = e_i,~ e_1e_2 = e_2e_1 = 0,\\
e_1x = \alpha x + \beta y = xe_2,~ xe_1 = (1-\alpha)x - \beta y = e_2x,\\
e_1y = \gamma x + (1-\alpha)y = ye_2,~ ye_1 = -\gamma x + \alpha y = e_2y,\\
xy = 2(\alpha e_1 + (1-\alpha)te_2),~ yx = -2((1-\alpha)e_1 + \alpha t e_2),\\
x^2 = -2\beta(e_1 - te_2),~ y^2 = 2\gamma(e_1 - te_2).
\end{gather*}
One can check that for all $\alpha, \beta, \gamma \in \mathbb{F}$, $U$ is a flexible superalgebra such that $U^{(+)} = D_t,$ thus, by Lemma \ref{flexplus} it is noncommutative Jordan. Denote this algebra by $D_t(\alpha, \beta, \gamma).$
Putting $t = -1,$ we obtain the superlagebra $M_{1,1}(\alpha, \beta, \gamma)$, and putting $t = -2$, we obtain the superalgebra $osp(1,2)(\alpha, \beta, \gamma)$ (see \cite{ps2}). Putting $\alpha = 1/2, \beta = \gamma = 0,$ we obtain a Jordan superalgebra $D_t.$

\medskip

We classify these algebras up to isomorphism:

\begin{Lem}
\label{Dt_class}
If $\mathbb{F}$ is a field which allows square root extraction, then for $\alpha, \beta, \gamma \in \mathbb{F}$, the superalgebra  $D_t(\alpha, \beta, \gamma)$ is isomorphic either to $D_t(\lambda,0,0):= D_t(\lambda)$ for some $\lambda \in \mathbb{F}$, or to $D_t(\frac{1}{2},\frac{1}{2},0).$
\end{Lem}

\begin{proof} Let $U = D_t(\alpha, \beta, \gamma)$ and consider the restriction of the operator $L_{e_1}$ to $U_1$. Since $\mathbb{F}$ allows square root extraction, $U_1$ has a Jordan basis with respect to $L_{e_1}.$ We denote the elements of this basis by $x' = ax + by, ~ y' = cx + dy, ~ a, b, c, d \in \mathbb{F}.$ We consider two cases:

$1) ~ x'L_{e_1} = \lambda x', ~ y'L_{e_1} = \mu y',$ where $\lambda, \mu \in \mathbb{F}.$ It is easy to see that $x' \circ x' = y' \circ y' = 0,$ and
$$x' \circ y' = \det \begin{pmatrix} a & b \\ c & d \end{pmatrix} (e_1 + te_2).$$ Hence, scaling if necessary, we may assume that the multiplication rules in the symmetrized superalgebra for $x', y'$ are the same as for $x, y.$ That is, we may assume that $x' = x$ and $y' = y$ are eigenvectors with respect to $L_{e_1}.$ Again using relation (\ref{p_i}), we see that $\lambda + \mu = 1.$ Repeating the calculations of the structure constants of $U$ as in the general case, we conclude that the original superalgebra is isomorphic to $D_t(\lambda, 0, 0).$

$2) ~ x'L_{e_1} = \lambda x' + y', ~ y'L_{e_1} = \lambda y',$ where $\lambda \in \mathbb{F}.$ Let
$$\delta = \sqrt{\det \begin{pmatrix} a & b \\ c & d \end{pmatrix}}.$$
Setting $x'' = x'/\delta, ~ y'' = y'/\delta,$ we see that the multiplication rules for $x, y$ and $x'', y''$ in the symmetrized algebra $D_t$ coincide. Thus, as in the previous case we may assume $x'' = x, y'' = y.$ From (\ref{p_i}) it follows that $\lambda = 1 - \lambda = 1/2.$ Repeating the calculations of the structure constants of $U$ as in the general case, we see that $D \cong D_t(1/2, 1, 0).$ Repeating these calculations for $\alpha = \beta = 1/2, \gamma = 0,$ we see that $D_t(1/2,1,0) \cong D_t(1/2,1/2,0)$ (the last algebra has more symmetrical multiplication rules and will be more convenient to work with). \end{proof}

\subsubsection{The superalgebra $K_3(\alpha, \beta, \gamma)$}
A noncommutative Jordan superalgebra $K_3(\alpha, \beta, \gamma) = U_{\bar{0}} \oplus U_{\bar{1}}; ~ U_{\bar{0}} = \langle e \rangle, U_{\bar{1}} = \langle z, w \rangle$ is defined by the following multiplication table:
\begin{table}[h]
\begin{tabular}{|l|l|l|l|}
\hline
    & $e$                     & $z$                  & $w$                      \\ \hline
$e$ & $e$                     & $\alpha z + \beta w$ & $\gamma z + (1-\alpha)w$ \\ \hline
$z$ & $(1-\alpha)z - \beta w$ & $-2 \beta e$         & $2 \alpha e$             \\ \hline
$w$ & $\alpha w - \gamma z$   & $-2(1- \alpha)e$     & $2 \gamma e$             \\ \hline
\end{tabular}
\end{table}

The superalgebra $K_3(\alpha, \beta, \gamma)^{(+)}$ is isomorphic to the simple nonunital Jordan superalgebra $K_3 = K_3(\frac{1}{2}, 0 , 0).$ In fact, they are characterized by this property:

\begin{Lem}[\cite{ps1}]
Let $U$ be a noncommutative Jordan superalgebra such that $U^{(+)} \cong K_3.$ Then $U \cong K_3(\alpha,\beta,\gamma)$ for some $\alpha,\beta,\gamma \in \mathbb{F}.$
\end{Lem}

Analogously to Lemma \ref{Dt_class} one can classify these algebras up to isomorphism:
\begin{Lem}
\label{K3_class}
If $\mathbb{F}$ is a field which allows square root extraction, then for $\alpha, \beta, \gamma \in \mathbb{F}$,   $K_3(\alpha, \beta, \gamma)$ is isomorphic either to $K_3(\lambda,0,0) = K_3(\lambda)$ for some $\lambda \in \mathbb{F}$, or to $K_3(\frac{1}{2},\frac{1}{2},0).$
\end{Lem}
Note that in fact $K_3(\alpha, \beta, \gamma)$ is of degree 1, but we still list these algebras here because of their similarity to the algebras $D_t(\alpha, \beta, \gamma)$ described above. In fact, the unital hull of $K_3(\alpha, \beta, \gamma)$ is a nonsimple noncommutative Jordan superalgebra $D_0(\alpha, \beta, \gamma).$

\medskip

\subsubsection{The superalgebra $U(V, f, \star)$}
Let $V = V_{\bar{0}} \oplus V_{\bar{1}}$ be a vector superspace over  $\mathbb{F},$ and let $f$ be a supersymmetric nondegenerate bilinear form on $V.$ Also let $\star$ be a superanticommutative multiplication on $V$ such that $f(x \star y, z) = f(x, y \star z)$ (that is, $f$ is an invariant form with respect to the product $\star$). Then we can define a multiplication on $U = \mathbb{F} \oplus V$ in the following way:
$$(\alpha + x)(\beta + y) = (\alpha \beta + f(x,y)) + (\alpha y + \beta x + x \star y),$$
and the resulting superalgebra (which is noncommutative Jordan) is denoted by $U(V, f, \star).$

The superalgebra $U(V, f, \star)^{(+)}$ is isomorphic to a simple Jordan superalgebra of nondegenerate supersymmetric bilinear form which is usually denoted by $J(V,f)$. Again, this property characterizes this family:

\begin{Lem}[\cite{ps1}]
Let $\mathbb{F}$ be a field which allows square root extraction, $J(V,f)$ be a superalgebra of nondegenerate supersymmetric bilinear form, and $U$ be a noncommutative Jordan superalgebra such that $U^{(+)}\cong J(V,f).$ Then there exists a superanticommutative product $\star$ on $V$ such that $U \cong U(V,f,\star).$
\end{Lem}

We remark here that the condition that $\mathbb{F}$ allows square root extraction serves to ensure that algebras $J(V,f)$ and $U(V,f,\star)$ have degree 2 (see the last section or \cite{ps1} for more details). If $\mathbb{F}$ does not allow square root extraction, it is possible that $J(V,f)$ and $U(V,f,\star)$ have degree 1. Note also that $J(V,f) = U(V,f,0).$ Examples of this type of algebras are generalized Cayley-Dickson algebras of dimension $2^n$, see \cite{kolca}.

\medskip

\subsubsection{The classification in degree 2}
Analogously to the examples described above, Pozhidaev and Shestakov calculated all possible generic Poisson brackets on simple Jordan superalgebras in the case of characteristic 0, concluding the classification in the case of degree $\geq 2.$ We provide it here:
\begin{Th}[\cite{ps2}]
\label{classification}
Let $U$ be a simple noncommutative central Jordan superalgebra over a field $\mathbb{F}$ of characteristic zero. Suppose that $U$ is neither supercommutative nor quasiassociative. Then $U$ is isomorphic to one of the following algebras: $K_3(\alpha, \beta, \gamma), D_t(\alpha, \beta, \gamma), J(\Gamma_n, A), \Gamma_n(\mathcal{D}),$ or there exists an extension $\mathbb{P}$ of $\mathbb{F}$ of degree $\leq 2$ such that $U \otimes_{\mathbb{F}} \mathbb{P} $ is isomorphic as a $\mathbb{P}$-superalgebra to $U(V, f, \star).$
\end{Th}

The only algebras in the list whose structure we did not mention here are the algebras $J(\Gamma_n, A).$ The underlying Jordan superalgebra of $J(\Gamma_n,A)$ is the \textit{Kantor double} of the Grassmann superalgebra $\Gamma_n$. The symmetrized superalgebra of $\Gamma_n(\mathcal{D})$ is the Grasmmann superalgebra $\Gamma_n$ (which is not simple and of degree 1). We will not treat these algebras in this paper. Their descriptions can be found in the paper \cite{ps2}. Note also that a lot of simple Jordan superalgebras, such as $P(2), K_{10}, K_9$, do not admit nonzero Poisson brackets and do not give new examples of simple algebras (see \cite{ps2}, \cite{ps3}).

In the positive characteristic some additional algebras appear, see \cite{ps3}.

\medskip

\subsection{Bimodules and representations}
In this subsection we briefly recall the basic notions of representation theory of nonassociative algebras.

\begin{Def} A \emph{superbimodule} over a superalgebra $A = A_{\bar{0}} \oplus A_{\bar{1}}$ is a linear superspace $M = M_{\bar{0}}\oplus M_{\bar{1}}$ with two bilinear operations $A \times M \rightarrow M, M \times A \rightarrow M$ such that $A_{\bar{i}}M_{\bar{j}} + M_{\bar{j}}A_{\bar{i}} + A_{\bar{j}}M_{\bar{i}} + M_{\bar{i}}A_{\bar{j}} \subseteq M_{\bar{i} + \bar{j}}$ for $\bar{i}, \bar{j} \in \mathbb{Z}_2.$
\end{Def}

For a subset $S \subset M$ by $\operatorname{Mod}(S)$ we denote the submodule generated by $S.$

\begin{Def}
A set $\{L,R\}$ of two even linear maps $L, R: A \to \operatorname{End}(M)$ is called a \emph{representation} of $A$.
\end{Def}

It is clear that the notions of superbimodule and representation are equivalent.

\begin{Def}
The \textit{regular superbimodule} $\operatorname{Reg}(A)$ for a superalgebra $A$ is defined on the vector space $A$ with the action of $A$ coinciding with the multiplication in $A.$
\end{Def}

\begin{Def}
For an $A$-superbimodule $M$ the structure of the \textit{opposite module} on the space $M^{op}$ with $M_{\bar{0}}^{op} = M_{\bar{1}}, M_{\bar{1}}^{op} = M_{\bar{0}}$ is defined by the action $a\cdot m = am, m \cdot a = (-1)^a ma$ for $a \in A, m \in M^{op}.$
\end{Def}

Note that a module $M$ is irreducible if and only if its opposite $M^{op}$ is irreducible.

\medskip

Recall the definition of a split null extension:
\begin{Def} The \textit{split null extension of $A$ by a module $M$} is a superalgebra $E = A \oplus M$ with the multiplication
\[(a_1+m_1)(a_2+m_2) = a_1a_2 + a_1\cdot m_2 + m_1\cdot a_2 \text{ for } a_1, a_2 \in A, m_1, m_2 \in M.\]
\end{Def}

Now we give the standard definition of a representation in a variety.

\begin{Def} Suppose that $A$ lies in a homogeneous variety of algebras $\mathfrak{M}$. Then an $A$-module $M$ is called an $\mathfrak{M}$\textit{-superbimodule} if the split null extension $E$ of $A$ by $M$ also lies in $\mathfrak{M}$. A representation of $A$ is called an $\mathfrak{M}$\textit{-represenation} if the corresponding superbimodule is an $\mathfrak{M}$-superbimodule.
\end{Def}

\begin{Lem}
Let $M$ be an $\mathfrak{M}$-superbimodule over $A$. Then $M^{op}$ is also an $\mathfrak{M}$-superbimodule over $A$.
\end{Lem}
\begin{proof} From the definition it follows that the split null extension $E = A \oplus M$ lies in $\mathfrak{M}.$ Consider the superalgebra $P = \langle 1, \bar{1} \rangle$ with $P_{\bar{0}} = \langle 1 \rangle, P_{\bar{1}} = \langle \bar{1} \rangle, 1$ is the unit of $P$, and $\bar{1}^2 = 0.$ It is easy to see that $P$ is an associative and supercommutative superalgebra, therefore, $E \otimes P \in \mathfrak{M}$. Note that $E \otimes P$ contains a subalgebra $E' = A \otimes 1 + M \otimes \bar{1}.$ One can easily check that $E'$ is isomorphic to the split null extension of $A$ by $M^{op}.$ Therefore, $M^{op}$ is a $\mathfrak{M}$-superbimodule over $A.$
\end{proof}

\subsection{Noncommutative Jordan representations}

Here we briefly recall the definitions of representation theory specific to noncommutative Jordan algebras. Then we formulate the definition of a noncommutative Jordan representation in the spirit of the Definition \ref{noncomm_jord_2mult}, and prove some technical results in this setting.

\medskip

Let $M$ be a noncommutative Jordan bimodule over $U$. Since $M \subseteq E$ is an ideal, $M$ has the Peirce decomposition with respect to $e$:
\[M = M_0 \oplus M_1 \oplus M_2,\]
where $M_i = M \cap E_i$ is the $i$th Peirce component of $M$.
Suppose now that $U$ is unital and $e = 1.$ Since $U \subseteq E_2(1)$, Peirce relations (\ref{pd_1}) imply that the Peirce components  $M_i(1)$ are submodules of $M.$

The relations (\ref{pd_1}) imply that $M_0$ is a \emph{zero bimodule}, that is, all $R_x, L_x, x \in U$ act as zero operators on it. On $M_2$ we have $L_1 = R_1 = id,$ so $M_2$ is called a \emph{unital bimodule}. On $M_1$ we have $R_1 + L_1 = id,$ such bimodules will be called \emph{special}.

\medskip

Definition \ref{noncomm_jord_2mult} states that $U$ can be considered as a Jordan superalgebra $(U,\circ)$ with generic Poisson bracket $\{\cdot,\cdot\}$. Here we state the definition of a noncommutative Jordan representation in this framework. The idea is clear: given a module $M$ over $U = (U,\cdot) = (U,\circ,\{\cdot,\cdot\}),$ we construct the split null extension $E = U \oplus M,$ then extend circle and the bracket products to $E$ and require that the algebra $(E,\circ,\{\cdot,\cdot\})$ be noncommutative Jordan. Finally, we express the conditions for a representation to be noncommutative Jordan by means of operator identities, similarly to (\ref{noncomm_jord_identity}). Here are the details:

\medskip

Let $A$ be a noncommutative Jordan superalgebra. For $x \in A$, introduce operators $R_x^+, R_x^- \in \operatorname{End}(A):$
\[R_x^+:= \frac{R_x + L_x}{2},~ R_x^-:= \frac{R_x - L_x}{2}.\]
These operators express circle and bracket multiplications:
\[yR_x^+ =y \circ x,~ yR_x^-= [y,x]/2,\]
where $x, y \in A$ (the denominator of 2 appearing in the formula seems unnatural for now, but it will make our future computations easier).

\medskip

Let $M$ be a noncommutative Jordan bimodule over $U$ and $E$ be the corresponding split null extension. Then for $x$ in $U$, $M$ is closed under operators $R_x^+, R_x^- \in \operatorname{End}(E).$ Since $M^2 = 0$ as a subalgebra of $E$, these operators should be understood as extending the circle and the generic Poisson bracket on $U$ to $E.$ Note that
\[L_x = \frac{R_x^+ - R_x^-}{2},~ R_x = \frac{R_x^+ + R_x^-}{2},\]
 so to give a structure of a noncommutative Jordan bimodule on a vector space $M$ it suffices to define the operators $R_x^+, R_x^-$ for every $x \in U.$

We can now state the new definition of a noncommutative Jordan representation:

\begin{Def} Let $U = (U,\circ,\{\cdot,\cdot\})$ be a noncommutative Jordan superalgebra, $M$ be a vector superspace, and $R^+, R^-:U \rightarrow \operatorname{End}(M)$ two even maps defined by $R^+:x \mapsto R_x^+, R^-:x \mapsto R_x^-$. Extend the multiplications on $U$ to the split null extension $E = U \oplus M$ as follows:
\[m \circ a = mR_a^+, \{m,a\} = 2mR_a^-, m \circ n = \{m,n\} = 0,\]
where $a \in U, m, n \in M.$ Then $M$ is a noncommutative Jordan superbimodule (equivalently, $\{R^+, R^-\}$ is a noncommutative Jordan representation) iff $(E, \circ, \{,\})$ is a noncommutative Jordan superalgebra.
\end{Def}

By Lemma \ref{JordPois} it is obvious that this definition is equivalent to the usual definition of a noncommutative Jordan superbimodule (noncommutative Jordan representation).

\medskip

Now we formulate the explicit conditions on the representation $\{R^+, R^-\}$ for it to be noncommutative Jordan.

Let $M$ be a noncommutative Jordan superbimodule over $U$ with the action given by $m \circ a = mR_a^+, \{m,a\} = 2mR_a^-$. By Lemma \ref{JordPois}, $(E,\circ)$ is a Jordan superalgebra. Thus, $M$ is a Jordan module over $J = U^{(+)}$ with the action given by $m\circ a = mR_a^{+}.$

The relation (\ref{pois_br}) is equivalent to two operator relations:
\begin{equation}
\label{r+r-1}
[R_a^{+},R_b^{-}] = \frac{1}{2}R^{+}_{[a,b]},
\end{equation}
\begin{equation}
\label{r+r-2}
R_a^{-}R_b^{+} + (-1)^{ab}R_b^{-}R_a^{+} = R^{-}_{a \circ b},
\end{equation}
where $a, b \in U.$
Thus, if $\{R^+,R^-\}$ is a noncommutative Jordan representation, then $R^+$ must be a Jordan representation and the two above relations must hold.
On the other hand, let $M$ be a module over $U$. Then from Lemma \ref{JordPois} and definitions of Jordan and noncommutative Jordan bimodule it follows that if $M$ is a Jordan bimodule over $U^{(+)}$ with the action given by $m \circ a = mR_a^+$ and relations (\ref{r+r-1}), (\ref{r+r-2}) hold, then the algebra $E$ is noncommutative Jordan, hence, the representation $R^+, R^-: U \rightarrow \operatorname{End}(M)$ is noncommutative Jordan. We state what we have just seen as a definition:

\begin{Def}
\label{2mult_rep}
Let $(U,\circ,\{\cdot,\cdot\})$ be a noncommutative Jordan superalgebra, and $M$ be a vector superspace. A representation $R^+, R^-: U \rightarrow \operatorname{End}(M)$ is noncommutative Jordan if and only if $R^+:U \rightarrow \operatorname{End}(M)$ is a Jordan representation of $U^{(+)}$ (that is, the relation (\ref{jord_identity}) holds for $R^+$) and relations (\ref{r+r-1}), (\ref{r+r-2}) hold.
\end{Def}

\medskip

It seems that the approach which we have just constructed works better with describing representations over superalgebras in which the circle and the bracket products are ''more natural'' than the usual multiplication (for example, noncommutative Jordan superalgebras that are explicitly built as a Jordan superalgebra with a generic Poisson bracket). For superalgebras in which the usual multiplication is more ''natural'' than the circle and bracket ones (for example, associative and quasiassociative ones), it appears better to stick with the usual definition of a noncommutative Jordan bimodule which was given in the beginning of the section (that is, to work with relations (\ref{noncomm_jord_identity}), (\ref{flex}), and the relations derived from them).

\medskip

We finish this subsection by proving some technical results that will be useful later. A large part of the subsequent calculations of the paper will be similar, so we note first that from now on we will use the Peirce relations (\ref{pd_1}), (\ref{pd_2}) in the operator form. For example, if $x \in U_0,$ then we can rewrite the relations (\ref{pd_1}) in the following way:
\[P_0L_x = P_0L_xP_0,~ P_0R_x = P_0R_xP_0,~ P_1L_x = P_1L_xP_1,~ P_1R_x = P_1R_xP_1,~ P_2L_x = 0,~ P_2R_x = 0.\]
We can transform analogously the relations (\ref{pd_1}), (\ref{pd_2}) for $x \in U_0, U_1, U_2$ and also with the operators $R^+, R^-$ instead of operators $L, R.$ The relations of this type we also call the Peirce relations.
\begin{Lem}
\label{r+r-rell}
Let $x \in U_1$. Then\\
$1)$ $P_0R_x^- = -P_0R_{[e,x]}^+;$\\
$2)$ $P_2R_x^- = P_2R_{[e,x]}^+;$\\
$3)$ $P_1R_x^-(P_0 + P_2) = P_1R_{[e,x]}^+(P_0 - P_2).$
\end{Lem}
\begin{proof} Consider the relation (\ref{r+r-1}) with $a = e, b = x: \frac{1}{2}R_{[e,x]}^+ = [R_e^+,R_x^-].$ Multiply it on $P_0$ on the left:
\[\frac{1}{2}P_0R_{[e,x]}^+ = P_0[R_e^+,R_x^-] = -P_0R_x^-R_e^+ = -P_0R_x^-P_1R_e^+ -\frac{1}{2}P_0R_x^,\]
 hence, the first relation follows. The second relation is obtained analogously by multiplying this relation on $P_2.$ Now multiply the same relation on $P_1$ on the left:
\[\frac{1}{2}P_1R_{[e,x]}^+ = P_1[R_e^+,R_x^-] = P_1R_x^-(\frac{1}{2}id - R_e^+) = \]
\[P_1R_x^-(P_0 + P_1 + P_2)(\frac{1}{2}id - R_e^+) = \frac{1}{2}P_1R_x^-(P_0 - P_2).\]
 Since by the Peirce relations $P_1R_{[e,x]}^+ = P_1R_{[e,x]}^+(P_0 + P_2),$ the third relation follows.
\end{proof}

\begin{Lem}
\label{r+r-rel}
Let $x, y, z \in U_1$ be such that $[x,e] =0,xy \in U_0 + U_2.$ Then

$1)$ $(P_0+P_2)R^{-}_x = 0,~ P_1R^{-}_x = P_1R^{-}_xP_1$;

$2)$ $P_1R^{+}_zR^{-}_x = 0;$

$3)$ $P_1R^{-}_xR^{+}_y = 0,~ P_1R^{+}_{[x,y]} = 0;$

$4)$ $P_1R^{-}_{x\circ y} = (-1)^{xy}P_1R_y^-R_x^+ = P_1R_x^+R_y^-.$ Particularly, if $[y,e] = 0,$ then $P_1R_{x \circ y}^- = 0.$
\end{Lem}

\begin{proof} The points $1)$ and $2)$ follow directly from the previous lemma.

$3)$ The relation (\ref{l_xp_1r_y+l_y}) implies that
\[P_1R_xP_1R_y^{+} = 0,~ P_1L_xP_1R_y^{+} = 0.\]
Subtracting the second equation from the first and using 1), we obtain $0 = P_1R_x^{-}P_1R_y^{+} = P_1R_x^{-}R_y^{+}.$ By 1) and the Peirce relations (\ref{pd_2}), we have $P_1R_y^+R_x^- = P_1R_y^+(P_0 + P_2)R_x^- = 0.$ Now the identity (\ref{r+r-1}) implies that $P_1R^{+}_{[x,y]} = 0.$

$4)$ The relation (\ref{r+r-2}) and the previous point imply that
\[P_1R^{-}_{x \circ y} = P_1(R^{-}_xR^{+}_y +(-1)^{xy} R_y^-R_x^+) = (-1)^{xy}P_1R_y^-R_x^+.\]
Now, the relation $P_1R_{[x,y]}^+ = 0$ implies $P_1(R_x^+R_y^- - (-1)^{xy}R_y^-R_x^+) = 0,$ and we have the second equality. The second statement follows from the point $2)$. \end{proof}

\section{Representations of simple noncommutative Jordan superalgebras of degree \texorpdfstring{$\geq 3$}{≥3}}

\medskip
In the following sections we study irreducible {\bf unital} representations over simple Jordan superalgebras. We begin with the case of degree $\geq 3$: we obtain results analogous to these of McCrimmon \cite{McC}. Basically, there are no new examples of noncommutative Jordan representations in this case.

From now on, if not explicitly said otherwise, by ''representation'' we mean a unital noncommutative Jordan representation, and by ''bimodule'' we mean a unital noncommutative Jordan bimodule.

\medskip

Representations of simple noncommutative Jordan superalgebras of degree $\geq 3$ are described in the following theorem (which is an analogue of McCrimmon's theorem from \cite{McC}):

\begin{Th}
\label{Repdeggeq3}
Let $U$ be a simple finite-dimensional separable noncommutative Jordan superalgebra of degree $\geq 3$ over the base field $\mathbb{F}$ which allows square root extraction. Then:\\
$1)$ If $U$ is Jordan, then every unital noncommutative Jordan superbimodule over $U$ is Jordan;\\
$2)$ If $U$ is quasiassociative, then for every unital noncommutative Jordan superbimodule $M$ over $U$ the split null extension $E = U \oplus M$ is quasiassociative;\\
$3)$ Every finite-dimensional unital noncommutative Jordan superbimodule over $U$ is completely reducible.
\end{Th}
\begin{proof} From Theorem \ref{class} and the classification of simple finite-dimensional associative and Jordan superalgebras it follows that $U$ has a unit. Let $e_1, \ldots, e_n$ be a system of pairwise orthogonal connected idempotents in $U$ which sum to 1. By Lemma \ref{ind_type} $U$ has the indicator $\phi = \lambda(1 - \lambda)$, and since $\mathbb{F}$ allows square root extraction, $\lambda \in \mathbb{F}$.  
Let $M$ be a noncommutative Jordan bimodule over $U.$ Since $U$ has $n \geq 3$ connected idempotents so does $E' = U + M.$ If the indicator $\phi$ of $U$ is $1/4$, then $E_2$ is commutative (thus, $U$ is also commutative) by Lemma \ref{ind1/4}. Hence, $M$ is a Jordan superbimodule over $U$ and is completely reducible by \cite{MZ}. If $\phi \neq 1/4$ we can mutate $E'$ to reduce to the case $\phi = 0$ as in the proof of the coordinatization theorem. Thus, by Lemma \ref{ind0}, $E_2$ is associative, and $M$ is an associative bimodule which is completely reducible by \cite{Pis}.

\end{proof}

In the next sections we study representations of simple noncommutative Jordan superalgebras of degree $\leq 2$.

\section{Representations of \texorpdfstring{$D_t(\alpha,\beta,\gamma) \text{ and } K_3(\alpha, \beta, \gamma)$}{Dt(α,β,γ) and K3(α,β,γ)}}

In this section we describe representations of the superalgebras $D_t(\alpha,\beta,\gamma)$ and simple nonunital superalgebras $K_3(\alpha, \beta, \gamma),$ except two cases:\\
$1)$ the case $\alpha = \frac{1}{2}, \beta = \gamma = 0$ is the case of Jordan superalgebra $D_t$ (resp., $K_3$), and will be dealt with later in the paper;\\
$2)$ the case $t = 1$, because in this case $D_t$ is of the type $U(V,f,\star).$ Indeed, its symmetrized superalgebra $D_1$ is a Jordan superalgebra of nondegenerate symmetric form on the space $V = \langle e_1 - e_2, x, y \rangle$ with $V_{\bar{0}} = \langle e_1 - e_2 \rangle, V_{\bar{1}} = \langle x, y \rangle.$ Representations of such superalgebras will be considered in subsequent papers.

\medskip

In what follows we assume that the base field $\mathbb{F}$ allows square root extraction. The Lemma \ref{Dt_class} then tells that for $\alpha, \beta, \gamma \in \mathbb{F}$, a superalgebra $D_t(\alpha, \beta, \gamma)$ belongs to one of the families: $D_t(\lambda,0,0)$ or $D_t(1/2,1/2,0),$ the first family consisting of ''almost associative'' (up to a mutation), and the second of ''almost commutative'' superalgebras. So we consider the two cases separately, using two different approaches given in the previous section. The results, however, are the same: except for some special values of parameters, every unital bimodule over a superalgebra $D_t(\alpha,\beta,\gamma)$ is completely reducible, with irreducible summands being the regular bimodule and its opposite. Also we classify irreducible representations over non-unital superalgebras $K_3(\alpha,\beta,\gamma).$ Note that in almost any case we make no dimensionality or characteristic restriction.

\medskip

As we said above, we study two cases separately.

\subsection{Representations of \texorpdfstring{$D_t(\lambda), \lambda \neq \frac{1}{2}$}{Dt(λ), λ ≠ 1/2 }}

In this subsection we classify all noncommutative Jordan representations over superalgebra $D_t(\lambda), \lambda \neq \frac{1}{2}, t \neq 1.$ First of all we describe a certain procedure, which we occasionally refer to as ''module mutation'', which in our case permits us to consider only the representations of the superalgebra $D_t(1)$, in which case the computations are drastically simplified.

\medskip

\textbf{Module mutation.} Let $U$ be a noncommutative Jordan (super)algebra and $M$ be a (super)bimodule over $U.$ Then the split null extension $E = U \oplus M$ is a noncommutative Jordan superalgebra. Let $\lambda \neq \frac{1}{2}$ be an element of the base field, and consider the $\lambda$-mutation $E^{(\lambda)},$ which is equal to $U^{(\lambda)}\oplus M.$ It is again a noncommutative Jordan superalgebra, and $M$ is an ideal of $E$ such that $M^2 = 0.$ Hence, we may consider $E^{(\lambda)}$ as the split null extension of $U^{(\lambda)}$ by the module $M$. Therefore, $M$ (with the action twisted by mutation) is a noncommutative Jordan bimodule over $U^{(\lambda)}.$ This construction is invertible: since $\lambda \neq \frac{1}{2}$, there exists $\mu \in \mathbb{F}$ such that $\lambda \odot \mu = 1.$ Mutating back again by $\mu,$ we obtain the original algebra $E$ (that is, we recover the original action of $U$ on $M$). Therefore, it is equivalent to study representations of a noncommutative Jordan superalgebra $U$ and any its nontrivial mutation $U^{(\lambda)}, \lambda \neq \frac{1}{2}.$ It is also clear that the module mutation preserves irreducibility and direct sum decomposition.

\medskip

Now we apply this construction to our case: if $\frac{1}{2} \neq \lambda \in \mathbb{F},$ then one can check that the $\mu$-mutation (where $\lambda \odot \mu = 1$) of $D_t(\lambda)$ is equal to $D_t(1).$ Hence, it suffices to study representations of the superalgebra $D_t(1).$

\medskip

For the reference, we provide the multiplication table of the algebra $D_t(1)= D$:
\begin{gather*}
D = D_{\bar{0}} \oplus D_{\bar{1}},~ D_{\bar{0}} = \langle e_1, e_2 \rangle,~ D_{\bar{1}} = \langle x, y \rangle,\\
e_i^2 = e_i,~ e_1e_2 = 0 = e_2e_1,\\
e_1x = x = xe_2,~ xe_1 = 0 = e_2x,~ e_1y = 0 = ye_2,~ e_2y = y = ye_1,\\
xy = 2e_1,~ yx = -2t e_2,~ x^2 = 0 = y^2.
\end{gather*}

The Peirce decomposition of $D$ relative to $e_1$ is the following:
\[D_0 = \langle e_2 \rangle, D_1 = \langle x, y \rangle, D_2 = \langle e_1 \rangle.\]
Now, let $M$ be a unital bimodule over $D$ and let $M = M_0 + M_1 + M_2$ be its Peirce decomposition with respect to $e_1.$ Our goal is to obtain enough operator relations derived from defining identities (\ref{noncomm_jord_identity}), (\ref{flex}) and Peirce relations in operator form, and then see that they in fact define completely the structure of a module over $D.$

\medskip

Apply Lemma \ref{lambda} to the split null extension $E = U \oplus M.$ Since $x \in D_1^{[1]}$, by point 1) of the Lemma \ref{lambda} we have
$(id-R_{e_1})R_x + L_{e_1}L_x = 0.$
Multiplying this relation on $P_1$ on the left, we get
\begin{equation}
\label{d1p1rx} P_1L_{e_1}(L_x + R_x) = 0.
\end{equation}
Multiplying the same relation on $P_0$ and $P_2$ on the left, we have
\begin{equation}
\label{d1p0rx} P_0R_x = 0,
\end{equation}
\begin{equation}
\label{d1p2lx} P_2L_x = 0.
\end{equation}
Now, by 2) of the Lemma \ref{lambda} we have
$L_x(id-L_{e_1}) + R_xR_{e_1} = 0.$
Multiplying this relation on $P_0$ and $P_2$ on the left and using relations (\ref{d1p0rx}), (\ref{d1p2lx}), we have
\begin{equation}
\label{d1p0lx} P_0L_xR_{e_1} = 0,~ P_0L_xL_{e_1} = P_0L_x,
\end{equation}
\begin{equation}
\label{d1p2rx} P_2R_xR_{e_1} = 0,~ P_2R_xL_{e_1} = P_2R_x.
\end{equation}

Analogously, since $y \in U_1^{[0]}$, by 1) and 2) of the Lemma \ref{lambda} we obtain the following relations:
\begin{equation}
\label{d1p0ly} P_0L_y = 0,
\end{equation}
\begin{equation*}
P_1R_{e_1}(R_y+L_y) = 0,
\end{equation*}
\begin{equation}
\label{d1p2ry}
P_2R_y = 0,
\end{equation}
\begin{equation}
\label{d1p0ry}
P_0R_yL_{e_1} = 0,~ P_0R_yR_{e_1} = P_0R_y,
\end{equation}
\begin{equation}
\label{d1p2ly} P_2L_yL_{e_1} = 0,~ P_2L_yR_{e_1} = P_2L_y.
\end{equation}

\medskip

Note that relations (\ref{d1p2lx}), (\ref{d1p0ly}) imply that
\begin{equation}
\label{d1pihom}
(P_0 + P_2)(L_{e_1}L_y - L_y) = 0,~ (P_0 + P_2)(L_{e_1}L_x) = 0.
\end{equation}

Combining the relations (\ref{d1p1rx}) and (\ref{d1p0lx}) with Peirce relations, we have
\begin{equation}
\label{d1p0lx(lx+rx)} P_0L_x(L_x + R_x) = P_0L_xP_1(R_{e_1} + L_{e_1})(L_x + R_x) = P_0L_xR_{e_1}(L_x + R_x) + P_0L_xP_1L_{e_1}(L_x + R_x) = 0.
\end{equation}

Analogously, we have
\begin{equation}
\label{d1p2rx(lx+rx)} P_2R_x(L_x + R_x) = 0,
\end{equation}
\[P_0R_y(L_y + R_y) = 0,~ P_2L_y(L_y + R_y) = 0.\]

Let $a = b = x$ in (\ref{flex_1}): $L_x^2 = R_x^2.$
Multiply this relation on $P_0$ on the left. Then (\ref{d1p0rx}) and (\ref{d1p0lx(lx+rx)}) imply that
\begin{equation}
\label{d1p0lx2}
P_0L_x^2 = 0,~ P_0L_xR_x = 0.
\end{equation}

Analogously, (\ref{d1p2lx}) and (\ref{d1p2rx(lx+rx)}) imply that
\begin{equation}
\label{d1p2rx2}
P_2R_x^2 = 0,~ P_2R_xL_x = 0.
\end{equation}

Analogously, substituting $a = b = y$ in (\ref{flex_1}), we obtain
\begin{equation}
\label{d1p0ry2}
P_0R_y^2 = 0,~ P_0R_yL_y = 0,
\end{equation}
\begin{equation}
\label{d1p2ly2}
P_2L_y^2 = 0,~ P_2L_yR_y = 0.
\end{equation}

\medskip

Let $a = e_1, b = x, c = y$ in (\ref{noncomm_jord_identity_1}):
\[2R_{e_1} - 2R_{e_1}(R_{e_1} + 2tR_{e_2}) + (R_x + L_x)(L_{e_1}L_y - L_y) - (R_y + L_y)(L_{e_1}L_x) = 0.\]
Multiply this relation on $P_1$ on the left:
\[P_1((2-2t)(R_{e_1}-R_{e_1}^2) + (R_x + L_x)(L_{e_1}L_y - L_y) - (R_y + L_y)(L_{e_1}L_x)) = 0.\]
By Peirce relations (\ref{pd_2}) $P_1(R_x + L_x)P_1 = P_1(R_y + L_y)P_1 = 0.$ Hence, by (\ref{d1pihom}) previous relation reduces to
\[(2-2t)P_1((R_{e_1}-R_{e_1}^2) = 0.\]
Since $t \neq 1,$ we have $P_1(R_{e_1} - R_{e_1}^2) = 0.$ Hence, $P_1R_{e_1} = (P_1R_{e_1})^2$ and $P_1L_{e_1} = (P_1L_{e_1})^2$ are orthogonal projections that sum to $P_1$. Hence, $M_1 = M_1^{[0]} + M_1^{[1]}.$ Further on we use this fact without mentioning it.

\medskip

Let $a = y, b = x$ in (\ref{flex_1}): $2tL_{e_2} - L_xL_y = 2R_{e_1} - R_xR_y.$
Multiplying this relation on Peirce projections on the left, using relations (\ref{d1p0rx}) and (\ref{d1p2lx}) and Peirce relations in operator form we have
\begin{equation}
\label{d1p0lxly} P_0L_xL_y = 2tP_0,
\end{equation}
\begin{equation}
\label{d1p12(1-t)re1}
P_1(2(1-t)R_{e_1} - R_xR_y + L_xL_y) = 0,
\end{equation}
\begin{equation}
\label{d1p2rxry} P_2R_xR_y = 2P_2.
\end{equation}

Analogously, substituting $a = x, b = y$ in (\ref{flex_1}), we have
\begin{equation}
\label{d1p0ryrx}
P_0R_yR_x = -2tP_0,
\end{equation}
\begin{equation}
\label{d1p12(1-t)le1} P_1(2(1-t)L_{e_1} - R_yR_x + L_yL_x) = 0,
\end{equation}
\begin{equation}
\label{d1p2lylx}
P_2L_yL_x = -2P_2.
\end{equation}

\medskip

Substituting $e = e_1, z \in \{x,y\}, w \in M_1$, and alternatively, $e = e_1, z \in M_1, w \in \{x,y\}$ in (\ref{p_i}), we obtain the following operator relations:
\begin{equation}
\label{p1rxp2p1lxp0}
P_1R_xP_2 = 0,~ P_1L_xP_0 = 0,
\end{equation}
\begin{equation}
\label{p1ryp0p1lyp2}
P_1R_yP_0 = 0,~ P_1L_yP_2 = 0.
\end{equation}

\medskip

The relation (\ref{l_xp_1r_y+l_y}) and the multiplication table of $D_t(1)$ imply that for $a, b \in U_1$ we have
\begin{equation}
\label{lap1lb+rb} P_1L_aP_1(L_b+R_b) = 0,~ P_1R_aP_1(L_b + R_b) = 0.
\end{equation}

\medskip

Consider the relation (\ref{d1p12(1-t)le1}):
\[0 = P_1(2(1-t)L_{e_1} - R_yR_x + L_yL_x) = \text{(by (\ref{p1ryp0p1lyp2}))} = 2(1-t)P_1L_{e_1} + P_1(L_y(P_1 + P_0)L_x - R_y(P_1 + P_2)R_x)=\]
\[P_1(2(1-t)L_{e_1} - R_yP_2R_x + L_yP_0L_x) + P_1(L_yP_1L_x - R_yP_1R_x).\]
Consider the second summand:
\[P_1(L_yP_1L_x - R_yP_1R_x) = \text{(since $P_1(R_y+L_y)P_1 = 0)$} = P_1L_yP_1(L_x+R_x) = \text{(by (\ref{lap1lb+rb}))} = 0.\]
Therefore, we have
\begin{equation}
\label{d1m1le1id}
2(1-t)P_1L_{e_1} =  P_1(R_yP_2R_x - L_yP_0L_x).
\end{equation}

Analogously, considering (\ref{d1p12(1-t)re1}) and using (\ref{p1rxp2p1lxp0}) and (\ref{lap1lb+rb}), we have
\begin{equation}
\label{d1m1re1id}
2(1-t)P_1R_{e_1} = P_1(R_xP_0R_y - L_xP_2L_y)
\end{equation}

\medskip

Multiply (\ref{d1p12(1-t)le1}) on the left by $P_0L_x$:
$$2(1-t)P_0L_xL_{e_1} = P_0L_x(R_yR_x - L_yL_x).$$
Combining this relation with (\ref{d1p0lx}) and (\ref{d1p0lxly}), we have $2P_0L_x = P_0L_xR_yR_x.$
Multiply this relation on $L_y$ on the right, and combine it with (\ref{d1p0lxly}):
\begin{equation}
\label{d1p0p2transf}
4tP_0 = 2P_0L_xL_y = P_0L_xR_yR_xL_y.
\end{equation}

The relation (\ref{d1p0lxly}) implies that $P_0L_xL_yP_1 = 0$. Hence, Peirce relations imply that
\[P_0L_xR_yP_1 = P_0L_xP_1R_yP_1 = -P_0L_xP_1L_yP_1 = 0.\]
Thus, from (\ref{p1ryp0p1lyp2}) it follows that
\begin{equation}
\label{d1p0lxry}
P_0L_xR_y = P_0L_xR_yP_2.
\end{equation}

\medskip

Now, multiply (\ref{d1p12(1-t)le1}) on the left by $P_2R_x$:
$2(1-t)P_2R_xL_{e_1} = P_2R_x(R_yR_x - L_yL_x).$
Combining this with (\ref{d1p2rx}) and (\ref{d1p2rxry}), we have
\begin{equation}
\label{d1p2rxlylx}
2tP_2R_x = P_2R_xL_yL_x.
\end{equation}
The relation (\ref{d1p2rxry}) implies that $P_2R_xR_yP_1 = 0$, hence, Peirce relations imply that $P_2R_xL_yP_1 = 0$. Hence, from (\ref{p1ryp0p1lyp2}) it follows that
\begin{equation}
\label{d1p2rxly}
P_2R_xL_y = P_2R_xL_yP_0.
\end{equation}

Analogously, multiplying the relation (\ref{d1p12(1-t)re1}) on the left by 
$P_2L_y$ and using (\ref{d1p2ly}) and (\ref{d1p2lylx}), we have
\begin{equation}
\label{d1p2lyrxry}
-2tP_2L_y = P_2L_yR_xR_y.
\end{equation}

\medskip

Consider the identity (\ref{flex}) for $a = x, b = y$:
$R_xL_y + L_yR_x = R_yL_x + L_xR_y.$
Multiplying it on the left by $P_0$ and $P_2$ and using correspondingly (\ref{d1p0rx}), (\ref{d1p0ly}) and (\ref{d1p2lx}), (\ref{d1p2ry}), we have
\begin{equation}
\label{d1p0rylx+lxry}
P_0(R_yL_x + L_xR_y) = 0.
\end{equation}
\begin{equation}
\label{d1p2rxly+lyrx}
P_2(R_xL_y + L_yR_x) = 0.
\end{equation}
Multiply (\ref{d1p2rxly+lyrx}) by $R_y$ on the right:
\begin{equation}
\label{d1p2rxlyry}
P_2R_xL_yR_y = -P_2L_yR_xR_y = \text{(by (\ref{d1p2lyrxry}))} = 2tP_2L_y.
\end{equation}

\medskip

Using the relations derived above we construct submodules of $M$ isomorphic to the regular module or its opposite.

\begin{Lem}
\label{Dt1-submod}
$1)$ Let $m$ be a homogeneous nonzero element of $M_2.$ Then
\[\operatorname{Mod}(m) = \langle m, mR_x, mL_y, mR_xL_y \rangle,\]
the elements $mR_x, mL_y \in M_1$ are linearly independent, $mR_xL_y \in M_0$ and
$$mR_xL_yL_x = 2tmR_x,~ mR_xL_yR_y = 2tmR_y.$$
If $t \neq 0,$ then $mR_xL_y \neq 0$ and $\operatorname{Mod}(m)$ is isomorphic to $\operatorname{Reg}(D)$ or its opposite depending on the parity of $m.$\\
$2)$ Suppose that $t \neq 0$ and let $m$ be a homogeneous nonzero element of $M_0.$ Then $\operatorname{Mod}(m)$ is isomorphic either to the regular $D$-superbimodule or its opposite.
\end{Lem}
\begin{proof}  $1)$ Let $0 \neq m$ be a homogeneous element of $M_2.$ The relations (\ref{d1p2lx}) and (\ref{d1p2ry}) imply that
$$mL_x = 0,~ mR_y = 0,$$
and relations (\ref{d1p2rxry}), (\ref{d1p2lylx}) imply that the elements $mR_x, mL_y \neq 0.$ Relations (\ref{d1p2rx}) and (\ref{d1p2ly}) imply that the elements $mR_x$ and $mL_y$ are linearly independent and that multiplying $mR_x$ and $mL_y$  by $e_1$ on both sides does not give new elements in $\operatorname{Mod}(m).$ From relations (\ref{d1p2rx2}) and (\ref{d1p2ly2}) it follows that
\[mR_x^2 = 0,~ mR_xL_x = 0,~ mL_y^2 = 0,~ mL_yR_y = 0.\]
From relations (\ref{d1p2rxry}) and (\ref{d1p2lylx}) we get
\[mR_xR_y = 2m = -mL_yL_x.\]
Relations (\ref{d1p2rxly+lyrx}) and (\ref{d1p2rxly}) imply that
\[mL_yR_x = -mR_xL_y \in M_0.\]
From relations (\ref{d1p0rx}), (\ref{d1p0ly}) we infer that
\[mR_xL_yR_x = 0,~mR_xL_yL_y = 0.\]
Relations (\ref{d1p2rxlylx}) and (\ref{d1p2rxlyry}) show respectively that
\[mR_xL_yL_x = 2tmR_x,~ mR_xL_yR_y = 2tmL_y.\]
Therefore, $\operatorname{Mod}(m)$ is equal to $\langle m, mL_y, mR_x, mR_xL_y \rangle,$ and $mR_xL_y \neq 0$ if $t \neq 0.$ Now, one can easily check that if $t \neq 0$, then $\operatorname{Mod}(m)$ is isomorphic to the regular $D_t(1)$-bimodule or its opposite. Indeed, one identifies
$$m \leftrightarrow e_1,~ mR_x \leftrightarrow x,~ mL_y \leftrightarrow y,~ mR_xL_y/2t \leftrightarrow e_2$$
if $m$ is even, and analogously if $m$ is odd.

\smallskip

$2)$ Let $0 \neq m \in M_0.$ The relation (\ref{d1p0p2transf}) implies that $\operatorname{Mod}(m) = \operatorname{Mod}(mL_xR_y),$ and relation (\ref{d1p0lxry}) implies that $mL_xR_y \in M_2.$ Hence, $\operatorname{Mod}(m)$ is generated by an element of $M_2$, therefore, by previous paragraph, it is either the regular $D_t(1)$-bimodule or its opposite.
\end{proof}

Now we are ready to describe the representations of $D_t(\lambda).$ First, we consider the case $t \neq 0:$

\begin{Th}
\label{Dt1-rep}
Let $M$ be a unital bimodule over $D_t(\lambda), t \neq 0,1, \lambda \neq \frac{1}{2}$. Then $M$ is completely reducible and the irreducible summands of $M$ are isomorphic to the regular $D_t(\lambda)$-bimodule or its opposite.
\end{Th}
\begin{proof} As we noted in the beginning of the section, module mutation preserves irreducibility and direct sum decomposition, so we may only consider the case $\lambda = 1.$ Since the regular bimodule over $D_t(1) = D$ and its opposite are irreducible, to prove the theorem it suffices to show that the submodule generated by any $m \in M$ is a sum of homomorphic images of the bimodules listed above. It is easy to see that
\[\operatorname{Mod}(m) = \operatorname{Mod}(mP_0,mP_1,mP_2) = \operatorname{Mod}(mP_0) + \operatorname{Mod}(mP_1) + \operatorname{Mod}(mP_2).\]
Hence, we can assume that $m$ is Peirce-homogeneous. Analogously, we can assume that $m \in M_{\bar{0}} \cup M_{\bar{1}}.$

If $m \in M_0 \cup M_2,$ the lemma above implies that $\operatorname{Mod}(m)$ is isomorphic either to $\operatorname{Reg}(D)$ or $\operatorname{Reg}(D)^{op}.$

Suppose that $m \in M_1.$ Since $M_1 = M_1^{[0]} \oplus M_1^{[1]}, \operatorname{Mod}(m) = \operatorname{Mod}(mL_{e_1}, mR_{e_1}).$ Thus, we can assume that $m \in M_1^{[0]} \cup M_1^{[1]}.$ If $m \in M_1^{[1]},$ then (\ref{d1m1le1id}) implies that $\operatorname{Mod}(m) = \operatorname{Mod}(mR_yP_2,mL_yP_0),$ therefore, $\operatorname{Mod}(m)$ can be generated by elements of $M_0$ and $M_2$, hence, it satisfies the above claim. Analogously, if $m \in M_1^{[0]},$ then (\ref{d1m1re1id}) implies that $\operatorname{Mod}(m) = \operatorname{Mod}(mR_xP_0, mL_xP_2)$.
\end{proof}

Now consider the case $t = 0.$ Note that $D_0(\lambda)$ is not simple: it contains an ideal $\langle e_1, x, y \rangle$, which is isomorphic to a simple nonunital superalgebra $K_3(\lambda,0,0).$ Hence, the regular bimodule over $D_0(\lambda)$ is not irreducible, but indecomposable. Therefore, in this case we restrict ourselves to describing irreducible bimodules.

\begin{Th}
\label{D01-rep}
Let $M$ be an irreducible bimodule over $D_0(\lambda), \lambda \neq \frac{1}{2}$. Then $M$ is isomorphic to one of the following bimodules:\\
$1)$ submodule of the regular $D_0(\lambda)$-bimodule generated by $e_1$, or its opposite;\\
$2)$ a one-dimensional space generated by an element $m$, with the only nonzero actions $mL_{e_2} = mR_{e_2} = m$.
\end{Th}

\begin{proof} Again we may only consider the case $\lambda = 1.$ Suppose that $M_2 \neq 0$ and let $0 \neq m$ be a homogeneous element of $M_2.$ By $1)$ of the Lemma \ref{Dt1-submod} $\operatorname{Mod}(m) = \langle m, mR_x, mL_y, mR_xL_y \rangle.$ Consider the element $mR_xL_y = - mL_yR_x \in M_0.$ Suppose that it is nonzero. Then relations (\ref{d1p0rx}), (\ref{d1p0ly}), (\ref{d1p2rxlylx}), (\ref{d1p2rxlyry}) imply respectively that
\[mR_xL_yR_x = 0,~ mR_xL_yL_y = 0,~ mR_xL_yL_x = 0,~ mR_xL_yR_y = 0.\]
Hence, $\operatorname{Mod}(mR_xL_y) = \langle mR_xL_y \rangle.$ But this submodule is neither zero nor the whole $M$, which contradicts irreducibility of $M$. Hence, $\operatorname{Mod}(m) = \langle m, mR_x, mL_y \rangle$ is isomorphic to the submodule of the regular $D_0(1)$-submodule generated by $e_1$ or its opposite, depending on the parity of $m.$ Indeed, one identifies
\[m \leftrightarrow e_1,~ mR_x \leftrightarrow x,~ mL_y \leftrightarrow y\]
if $m$ is even, and analogously if $m$ is odd.

\smallskip

Suppose now that $M_2 = 0, M_0 \neq 0,$ and let $0 \neq m$ be a homogeneous element of $M_0.$ Relations (\ref{d1p0rx}), (\ref{d1p0ly}) imply that $mR_x = 0, mL_y = 0.$ Relations (\ref{d1p0lx2}), (\ref{d1p0lxly}), (\ref{d1p0lxry}) and (\ref{d1p0lx}) show that
\[mL_x^2 = 0,~ mL_xR_x = 0,~ mL_xL_y = 0,~ mL_xR_y = 0,~ mL_xL_{e_1} = mL_x,~ mL_xR_{e_1} = 0.\]
This implies that $\operatorname{Mod}(mL_x) = \langle mL_x \rangle.$ Since $m \notin \langle mL_x \rangle$, $mL_x = 0.$ Relations (\ref{d1p0ry2}), (\ref{d1p0ryrx}), (\ref{d1p0rylx+lxry}) and (\ref{d1p0ry}) imply respectively that
\[mR_y^2 = 0,~ mR_yL_y = 0,~ mR_yR_x = 0,~ mR_yL_x = 0,~ mR_yL_{e_1} = 0,~ mR_yR_{e_1} = mR_y.\]
Again, this means that $mR_y = 0.$ Then $\operatorname{Mod}(m) = M = \langle m \rangle$, and $D_0(1)$ acts on $M$ in the following way: $mL_{e_2} = mR_{e_2} = m,$ and all other elements act as zero operator. This bimodule is just $\operatorname{Reg}(D_0(1))/\langle e_1, x, y  \rangle,$ and is obviously noncommutative Jordan.

\smallskip

Now suppose that $M_0 = M_2 = 0.$ Hence, $P_0 = 0, P_2 = 0.$ Sum the relations (\ref{d1m1le1id}) and (\ref{d1m1re1id}):
\[2P_1 = P_1(R_yP_2R_x - L_yP_0L_x + R_xP_0R_y - L_xP_2L_y).\]
From this relation it follows that $P_1 = 0,$ therefore, $M = M_1 = 0.$
\end{proof}

\subsection{Representations of \texorpdfstring{$D_t(\frac{1}{2},\frac{1}{2},0)$}{Dt(1/2,1/2,0)}}
$\phantom{jkl}$\\
Let $D= D_t(1/2,1/2,0).$ This algebra is very close to being commutative (see the multiplication table below). Therefore, we study its representations using the approach given by the Definitions \ref{noncomm_jord_2mult} and \ref{2mult_rep}, that is, interpreting it as a superalgebra with Jordan and bracket products and using the $R^+$ and $R^-$ operators.

\medskip

We start with the multiplication table for $D:$
\begin{gather*}e_1^2 = e_1,~ e_2^2 = e_2,~ e_1 \circ e_2 = [e_1,e_2] = 0,\\
e_1 \circ x = e_2 \circ x = x/2,~ [e_1,x] = -[e_2,x] = y,\\
e_1 \circ y = e_2 \circ y = y/2,~ [e_1,y] = [e_2,y] = 0,\\
[x,x] = -2(e_1-te_2),~ [y,y] = 0,\\
x \circ y = e_1 + te_2,~ [x,y] = 0.
\end{gather*}

Fix an idempotent $e = e_1 \in D.$ Then $D_0 = D_0(e_1) = \langle e_2 \rangle, D_1 = D_1(e_1) = \langle x, y \rangle, D_2 = D_2(e_1) = \langle e_1 \rangle.$ Let $M$ be a unital bimodule over $D$, and let $M = M_0 \oplus M_1 \oplus M_2$ be its Peirce decomposition with respect to $e_1.$

First we prove the following proposition that will allow us to reduce the study of $D_t(\frac{1}{2},\frac{1}{2},0)$-bimodules to the study of Jordan bimodules over $D_t.$
\begin{proposition}
\label{d0r-}
$1)$ The operators $R_a^-, a \in D$ lie in the enveloping associative algebra of the Jordan representation $R^+: D^{(+)} \rightarrow \operatorname{End}(M).$ The expressions for $R^-$ operators of basis elements $e_1, e_2, x, y$ do not depend on $M;$\\
$2)$ A subspace $N \subseteq M$ is a $D$-submodule if and only if it is a submodule with respect to the representation $R^+;$\\
$3)$ $M$ is irreducible if and only if it is irreducible with respect to $R^+;$\\
$4)$ Two bimodules over $D$ are isomorphic if and only if they are isomorphic as Jordan bimodules over $D^{(+)}$ with the symmetrized action.
\end{proposition}
\begin{proof} To prove 1), we need to express the operators $R_a^-, a \in D$ as polynomials in operators $R_a^+, a \in D.$ By Lemma \ref{r+r-rel} we have
\begin{equation}
\label{d0ry-}
(P_0 + P_2)R_y^- = 0,~ P_1R_y^-P_1 = P_1R_y^-.
\end{equation}

By Lemma \ref{r+r-rell} we have
\begin{equation}
\label{d0p0rx-}
P_0R_x^- = - P_0R_y^+,
\end{equation}
\begin{equation}
\label{d0p2rx-}
P_2R_x^- = P_2R_y^+,
\end{equation}
\begin{equation}
\label{d0p1rx-(p0+p2)}
P_1R_x^-(P_0 + P_2) = P_1R_y^+(P_0 - P_2).
\end{equation}

Note that since $[x,y] = [y,y] = 0$ and $[x,x] = -2(e_1-te_2)$, from (\ref{l_xp_1r_y+l_y}) it follows that
\begin{equation}
\label{d0p1ra-p1rb+}P_1R_a^-P_1R_b^+ = 0, \text{ where } a, b \in D_1.
\end{equation}

\medskip

Substituting $a = y, b = y$ in (\ref{r+r-1}), we have $[R_y^+,R_y^-] = 0.$ Multiplying this relation by Peirce projections on the left and using (\ref{d0ry-}) and Peirce relations, one can see that
\begin{equation}
\label{d0(p0+p2)ry+ry-}
R_y^+R_y^- = 0,~ R_y^-R_y^+ = 0.
\end{equation}

\medskip

Substituting $a = y, b = x$ in (\ref{r+r-1}), we have $[R_y^+,R_x^-] = 0.$ Multiply this relation by $P_0 + P_2$ on the left and by $P_1$ on the right:
\begin{equation}
\label{d0p0ry+rx-p1}
0 = (P_0+P_2)(R_y^+R_x^- + R_x^-R_y^+)P_1 = (P_0+P_2)R_y^+R_x^-P_1.
\end{equation}

\medskip

Substituting $a = y, b = x$ in (\ref{r+r-2}), we have $R_y^-R_x^+ - R_x^-R_y^+ = -R_{e_1 + te_2}^-.$ Multiplying this relation by $P_0$ and $P_2$ on the left and using (\ref{d0ry-}), (\ref{d0p0rx-}), (\ref{d0p2rx-}), we have
\[0 = P_0R_x^-R_y^+ = -P_0(R_y^+)^2,\]
\[0 = P_2R_x^-R_y^+ = P_2(R_y^+)^2,\]
thus,
\begin{equation}
\label{d0(p0+p2)ry+2}
(P_0+P_2)(R_y^+)^2 = 0.
\end{equation}

\medskip

Substituting $a = x, b = y$ in (\ref{r+r-1}), we have $[R_x^+,R_y^-] = 0.$ Multiplying this relation by Peirce projections on the left and using (\ref{d0ry-}), one has
\begin{equation}
\label{d0rx+ry-}
R_x^+R_y^- = 0, R_y^-R_x^+ = 0.
\end{equation}

\medskip

Substituting $a = x, b = x$ in (\ref{r+r-1}), we have
\begin{equation}
\label{d0rx+rx-}
[R_x^+,R_x^-] = -R_{e_1 - te_2}^+.
\end{equation}
Multiply this relation by $P_0$ on the left and by $P_0 + P_1$ on the right:
\[ tP_0 = P_0(R_x^+R_x^- + R_x^-R_x^+)(P_0 + P_1) = \text{(by (\ref{d0p0rx-}))} = P_0(R_x^+P_1R_x^-(P_0 + P_1) - R_y^+R_x^+P_0) = \]
\[= \text{(by (\ref{d0p1rx-(p0+p2)}))} = P_0((R_x^+R_y^+ - R_y^+R_x^+)P_0 + R_x^+R_x^-P_1),\]
hence,
\begin{equation}
\label{d0p0(rx+ry+-ry+rx+)}
P_0(R_x^+R_y^+ - R_y^+R_x^+)P_0 = tP_0,
\end{equation}
\begin{equation}
\label{d0p0rx+rx-p1}
P_0R_x^+R_x^-P_1 = 0.
\end{equation}

Analogously, multiplying the relation (\ref{d0rx+rx-}) by $P_2$ on the left and by $P_1 + P_2$ on the right, we have
\begin{equation}
\label{d0p2rx+rx-p1}
P_2R_x^+R_x^-P_1 = 0,
\end{equation}
\begin{equation}
\label{d0p2(rx+ry+-ry+rx+)}
P_2(R_x^+R_y^+ - R_y^+R_x^+)P_2 = P_2.
\end{equation}

Multiplying the relation (\ref{d0rx+rx-}) by $P_1$ on the left, we have
\[-\frac{1-t}{2}P_1 = P_1(R_x^+R_x^- + R_x^-R_x^+)=\text{(by (\ref{d0p1ra-p1rb+}))}=  P_1(R_x^+(P_0 + P_2)R_x^- + R_x^-(P_0 + P_2)R_x^+)=\]
\[ = \text{(by (\ref{d0p0rx-}), (\ref{d0p2rx-}), (\ref{d0p1rx-(p0+p2)}))} = P_1(R_x^+(P_2 - P_0)R_y^+ + R_y^+(P_0 - P_2)R_x^+).\]
Therefore (remember that $t \neq 1$),
\begin{equation}
\label{d0p1}
P_1 = \frac{2}{1-t}P_1(R_x^+(P_0 - P_2)R_y^+ + R_y^+(P_2 - P_0)R_x^+).
\end{equation}

\medskip

Multiply (\ref{d0p1}) by $R_y^-$ on the right:
\begin{equation}
\label{d0p1ry-}
P_1R_y^- =\frac{2}{1-t}P_1(R_x^+(P_0 - P_2)R_y^+R_y^- + R_y^+(P_2 - P_0)R_x^+R_y^-) = \text{ (by (\ref{d0(p0+p2)ry+ry-}), (\ref{d0rx+ry-}))} = 0.
\end{equation}

\medskip

Multiply (\ref{d0p1}) by $R_{e_1}^-$ on the right:
$$P_1R_{e_1}^- = \frac{2}{1-t}P_1(R_x^+(P_0 - P_2)R_y^+R_{e_1}^- + R_y^+(P_2 - P_0)R_x^+R_{e_1}^-) = \text{ (by \ref{r+r-1})} =$$
\begin{equation}
\label{d0p1re1-}
= \frac{2}{1-t}P_1(R_x^+(P_0 - P_2)R_{e_1}^-R_y^+ + R_y^+(P_2 - P_0)(R_{e_1}^-R_x^+ - \frac{1}{2}R_y^+)) = \frac{1}{1-t}P_1R_y^+(P_0 - P_2)R_y^+.
\end{equation}

\medskip

Multiply (\ref{d0p1}) on $R_x^{-}P_1$ on the right:
$$P_1R_x^-P_1 = \frac{2}{1-t}P_1(R_x^+(P_0 - P_2)R_y^+R_x^-P_1 + R_y^+(P_2 - P_0)R_x^+R_x^-P_1) = \text{ (by (\ref{d0p0ry+rx-p1}), (\ref{d0p0rx+rx-p1}), (\ref{d0p2rx+rx-p1}))} =  0.$$
Therefore,
\begin{equation}
\label{d0p1rx-}
P_1R_x^- = P_1R_x^-(P_0 + P_2) = \text{ (by (\ref{d0p1rx-(p0+p2)}))} = P_1R_y^+(P_0 - P_2).
\end{equation}

Now, one can see that
\begin{gather*}
R_x^- = (P_0 + P_1 + P_2)R_x^- = \text{(by (\ref{d0p0rx-}), (\ref{d0p2rx-}), (\ref{d0p1rx-}))} = -P_0R_y^+ + P_1R_y^+(P_0 - P_2) + P_2R_y^+,\\
R_{e_1}^- = -R_{e_2}^- = P_1R_{e_1}^- = \text{(by (\ref{d0p1re1-}))} = \frac{1}{1-t}P_1R_y^+(P_0 - P_2)R_y^+,\\
R_y^- = (P_0+P_1+P_2)R_y^- = \text{(by (\ref{d0ry-}), (\ref{d0p1ry-}))} = 0.
\end{gather*}

These relations and the fact that Peirce projections $P_i$ are polynomials in $R_{e_1}^+$ imply that
the operators $R_a^-, a \in D$ lie in the enveloping associative algebra of the Jordan representation $R^+: D \rightarrow \operatorname{End}(M).$ Also it is clear that the operators $R_{e_1}^-, R_{e_2}^-, R_x^-, R_y^-$ do not depend on the module $M.$ Therefore, point $1)$ is now proved. It follows that the structure of $M$ as a noncommutative Jordan superbimodule is completely determined by its structure as a Jordan superbimodule over $D^{(+)}.$ The other points follow immediately from this statement.
\end{proof}

Consider first the case $t \neq -1.$ In this case we have the following result:
\begin{Lem}\cite{Trush}
\label{sl2}
Let $t \neq -1.$ Then the operators
\[E = \frac{2}{1+t}(R_x^+)^2, F = \frac{2}{1+t}(R_y^+)^2, H = \frac{2}{1+t}(R_x^+R_y^+ + R_y^+R_x^+)\]
span the simple Lie algebra $\mathfrak{sl}_2$, that is, $[E,H] = 2E,~[F,H] = -2F,~[E,F] = H.$
\end{Lem}
From Peirce relations it follows that $M_0 + M_2$ is invariant under $E, F$ and $H.$ Hence, (\ref{d0(p0+p2)ry+2}) and the multiplication table of $\mathfrak{sl}_2$ imply that
\begin{equation}
\label{d0(m0+m2)sl2}
(M_0 + M_2)(R_x^+)^2 = 0,~ (M_0 + M_2)(R_y^+)^2 = 0,~ (M_0 + M_2)(R_x^+R_y^+ + R_y^+R_x^+) = 0.
\end{equation}

\medskip

As in the previous subsection, with the aid of the relations above we can find submodules in $M$ that are isomorphic to the regular one or its opposite:
\begin{Lem}
\label{d0-submod}
Let $t \neq -1, 1.$\\
$1)$ Let $m$ be a nonzero homogeneous element in $M_2.$ Then
\[\operatorname{Mod}(m) = \langle m, mR_x^+, mR_y^+, mR_x^+R_y^+P_0 \rangle.\]
Also $mR_x^+$ and $mR_y^+$ are linearly independent,
\[m(R_x^+)^2 = m(R_y^+)^2 = 0,~ mR_x^+R_y^+P_0R_x^+ = tmR_x^+/2,~ mR_x^+R_y^+P_0R_y^+ = tmR_y^+/2,\]
and, if $t \neq 0,$ $\operatorname{Mod}(m)$ is isomorphic to $\operatorname{Reg}(D)$ or $\operatorname{Reg}(D)^{op}$ depending on the parity of $m;$\\
$2)$ Let $t \neq 0$ and let $m$ be a nonzero homogeneous element in $M_0.$ Then $\operatorname{Mod}(m)$ is isomorphic to $\operatorname{Reg}(D)$ or $\operatorname{Reg}(D)^{op}$ depending on the parity of $m.$
\end{Lem}
\begin{proof} 1) Relations (\ref{d0(m0+m2)sl2}) and (\ref{d0p2(rx+ry+-ry+rx+)}) imply that
\[m(R_x^+)^2 = m(R_y^+)^2 = 0, mR_x^+R_y^+P_2 = -mR_y^+R_x^+P_2 = m/2.\]
Hence, $mR_x^+, mR_y^+ \neq 0.$ Suppose that $mR_x^+ = \alpha mR_y^+$ for some $\alpha \in \mathbb{F}$. Acting by $R_x^+P_2$ on this relation and using (\ref{d0(m0+m2)sl2}) we see that $-\frac{\alpha}{2}m = 0,$ hence, $\alpha = 0,$ a contradiction. Therefore, $mR_x^+$ and $mR_y^+$ are linearly independent.

Consider the identity that holds for every Jordan superalgebra:
\[R_{(a,b,c)} = - (-1)^{bc}[[R_a,R_c],R_b].\]
Consider this identity for the symmetrized superalgebra of the split null extension $E^{(+)} = (D\oplus M)^{(+)}.$ Substituting in it $a = x, b = y, c = x$  we get $\frac{1+t}{2}R_x^+ = [(R_x^+)^2,R_y^+].$ Applying this relation on $m$ and using (\ref{d0(m0+m2)sl2}), we get
\[m[(R_x^+)^2,R_y^+] = -mR_y^+(R_x^+)^2 = mR_x^+R_y^+(P_0 + P_2)R_x^+ = \frac{m}{2}R_x^+ + mR_x^+R_y^+P_0R_x^+ = \frac{1+t}{2}mR_x^+.\]

Denote $n = mR_x^+R_y^+P_0.$ Then the relation above implies that $nR_x^+ = \frac{t}{2}mR_x^+.$
Analogously, substituting $a = y, b = x, c = y$ and applying the resulting relation on $m$, we get $nR_y^+ = \frac{t}{2}mR_y^+.$ Thus, the space $\langle m, mR_x^+, mR_y^+, n \rangle$ is closed under all $R^+$ operators. Proposition \ref{d0r-} implies that it is also closed under all $R^-$ operators and $\langle m, mR_x^+, mR_y^+, n \rangle = \operatorname{Mod}(m).$

Suppose that $t \neq 0.$ Then, since $nR_x^+ = tmR_x^+, n \neq 0.$ Hence, $\operatorname{Mod}(m)$ is isomorphic as a Jordan superbimodule to the regular $D^{(+)}$-bimodule or its opposite. Indeed, one identifies
\[m \leftrightarrow e_1,~ mR_x^+ \leftrightarrow x/2,~ mR_y^+ \leftrightarrow y/2,~ n \leftrightarrow te_2/2\]
if $m$ is even, and analogously if $m$ is odd. Therefore, by 4) of Proposition \ref{d0r-} it is isomorphic to the regular $D$-bimodule or its opposite.

\smallskip

2) Analogously to the point 1), from relations (\ref{d0(m0+m2)sl2}) and (\ref{d0p0(rx+ry+-ry+rx+)}) it follows that
\[mR_x^+, mR_y^+ \neq 0,~ mR_x^+R_y^+P_0 = \frac{t}{2}m.\]
Consider the element $n = mR_x^+R_y^+P_2.$ Analogously to the previous point, we get $nR_x^+ = mR_x^+,$ therefore, $nR_x^+R_y^+P_0 = \frac{t}{2}m.$ Hence, $\operatorname{Mod}(m) = \operatorname{Mod}(n)$ and everything follows from the previous point.
\end{proof}

Now we are ready to describe noncommutative Jordan bimodules over $D.$

\begin{Th}
Let $M$ be an superbimodule over $D_t(\frac{1}{2}, \frac{1}{2}, 0) = D, t \neq -1, 0, 1$. Then $M$ is completely reducible and its irreducible summands are isomorphic either to $\operatorname{Reg}(D)$ or $\operatorname{Reg}(D)^{op}.$
\end{Th}

\begin{proof} It is enough to show that every one-generated bimodule is a sum of homomorphic images of $\operatorname{Reg}(D)$ and $\operatorname{Reg}(D)^{op}.$ Let $m \in M$ and consider $\operatorname{Mod}(m).$ As in the previous subsection, we may suppose that $m$ is homogeneous and Peirce-homogeneous.

If $m \in M_0 \cup M_2$, from the prervious lemma it follows that $\operatorname{Mod}(m)$ is isomorphic as a noncommutative Jordan superbimodule to the regular $D$-bimodule or its opposite.

Suppose that $m \in M_1.$ Then relation (\ref{d0p1}) implies that $\operatorname{Mod}(m)$ can be generated by homogeneous elements of $M_0 + M_2.$ Hence, $\operatorname{Mod}(m)$ is a sum of bimodules isomorphic to $\operatorname{Reg}(D)$ or $\operatorname{Reg}(D)^{op}.$
\end{proof}

Consider now the case $t = 0.$ Since in this case $D$ is not simple, we will restrict ourselves to describing irreducible $D$-bimodules.
\begin{Th}
\label{dt1/2-rep}
Let $M$ be an irreducible bimodule over $D_0(\frac{1}{2}, \frac{1}{2}, 0) = D$. Then $M$ is isomorphic to one of the following bimodules:\\
$1)$ submodule of the regular $D$-bimodule generated by $e_1$, or its opposite;\\
$2)$ a one-dimensional space generated by an element $m$, with the only nonzero action $mR_{e_2}^+ = m$.
\end{Th}

\begin{proof} Suppose that $M_2 \neq 0$ and take a homogeneous $0 \neq m \in M_2.$ From Lemma \ref{d0-submod} it follows that $\operatorname{Mod}(m) = \langle m, R_x^+, R_y^+, mR_x^+R_y^+P_0 \rangle$ and for $n = mR_x^+R_y^+P_0$ we have $nR_x^+ = 0, nR_y^+ = 0.$ Therefore, by $2)$ of Proposition \ref{d0r-}, $\operatorname{Mod}(n) = \langle n \rangle.$ Since $m \notin \operatorname{Mod}(n), n = 0.$ Thus, $M = \operatorname{Mod}(m) = \langle m, mR_x^+, mR_y^+ \rangle$ is isomorphic as a Jordan superbimodule to the submodule of the regular $D_0$-bimodule generated by $e_1$ or its opposite: explicitly, one identifies
\[m \leftrightarrow e_1,~ mR_x^+ \leftrightarrow x/2,~ mR_y^+ \leftrightarrow y/2\]
if $m$ is even, and analogously if $m$ is odd. Hence, from the Proposition \ref{d0r-} it follows that $M$ is isomorphic as a noncommutative Jordan bimodule to the submodule of the regular $D$-bimodule generated by $e_1$ or its opposite.

Now suppose that $M_2 = 0, M_0 \neq 0$ and take a homogeneous $0 \neq M_0.$ Relations (\ref{d0(m0+m2)sl2}), (\ref{d0p0(rx+ry+-ry+rx+)}) and the fact that $M_2 = 0$ imply that
\[m(R_x^+)^2 = 0,~ m(R_y^+)^2 = 0,~ mR_x^+R_y^+ = 0,~ mR_y^+R_x^+ = 0.\]
Therefore,
\[\operatorname{Mod}(mR_x^+) = \langle mR_x^+ \rangle,~ \operatorname{Mod}(mR_y^+) = \langle mR_y^+ \rangle.\]
Since $m$ does not belong to any of these two submodules, $mR_x^+ = mR_y^+ = 0,$ which implies that $M = \operatorname{Mod}(m) = \langle m \rangle,$ and the only nonzero action is $mR_{e_2}^+ = m.$ It is easy to see that $M = \operatorname{Reg}(D)/\langle e_1, x, y \rangle,$ therefore, it is obviously noncommutative Jordan.

Suppose now that $M_0 = M_2 = 0.$ Then relation (\ref{d0p1}) implies that $M_1 = 0$ and $M = 0.$
\end{proof}

Consider the case $t = -1$. In this case we only describe finite-dimensional irreducible superbimodules over $D_{-1}(\frac{1}{2},\frac{1}{2},0) = D.$ Also in this case we assume that $\mathbb{F}$ is algebraically closed.

\medskip

For $\alpha, \beta, \gamma \in \mathbb{F}$ consider a superbimodule $V(\alpha,\beta,\gamma)$ over $D^{(+)} = J$ with $V_{\bar{0}} = \langle v, w \rangle, V_{\bar{1}} = \langle z, t \rangle$ and the multiplication table
\begin{gather*}v \circ e_1 = v, w \circ e_1 = 0, z \circ e_1 = \frac{z}{2}, t \circ e_1 = \frac{t}{2},\\
v \circ e_2 = 0,~ w \circ e_2 = w,~ z \circ e_2 = \frac{z}{2},~ t \circ e_2 = \frac{t}{2},\\
v \circ x = z,~ w \circ x = (\gamma-1)z - 2\alpha t,~ z \circ x = \alpha v,~ t \circ x = \frac{1}{2}((\gamma-1)v-w),\\
v \circ y = t,~ w \circ y = 2 \beta z - (\gamma+1)t,~ z\circ y = \frac{1}{2}((\gamma+1)v+w),~ ty = \beta w.
\end{gather*}

In the paper \cite{MZ} it was proved that the modules $V(\alpha,\beta,\gamma)$ are Jordan and every finite-dimensional irreducible Jordan superbimodule over $J$ is isomorphic either to $V(\alpha, \beta, \gamma)$ (if $\gamma^2 - 4\alpha\beta - 1 \neq 0$), or $V_1 = \langle w, w\circ J_{\bar{1}} \rangle$, or $V_2 = V/V_1$ (if $\gamma^2 - 4\alpha\beta - 1 = 0$).

Therefore, by Proposition \ref{d0r-}, we have to check whether bimodules $V(\alpha,\beta,\gamma)$ admit a structure of noncommutative Jordan bimodule over $D.$

\medskip

Suppose that $V = V(\alpha,\beta,\gamma)$ admits a structure of a noncommutative Jordan bimodule over $D.$ Note that with respect to $e_1$ the Peirce decomposition of $V$ is the following:
\[V_0 = \langle w \rangle, V_1 = \langle z,t \rangle, V_2 = \langle v \rangle.\]
Note also that the operators $(R_x^+)^2, (R_y^+)^2, R_x^+R_y^+ + R_y^+R_x^+$ act on $V(\alpha,\beta,\gamma)$ as $\alpha, \beta, \gamma$, respectively. Hence, the relation (\ref{d0(p0+p2)ry+2}) implies that $\beta = 0.$

The proof of Proposition \ref{d0r-} shows that there is only one way to introduce the noncommutative Jordan action of $D$ in $V(\alpha,0,\gamma):$
\begin{gather*}
wR_x^- = -wR_y^+ = (\gamma+1)t,~ vR_x^- = vR_y^+ = t,\\
zR_x^- = zR_y^+(P_0+P_2) = \frac{1}{2}(w - (\gamma - 1)v),\\
zR_{e_1}^- = \frac{1}{2}P_1R_y^+(P_0-P_2)R_y^+ = -\frac{\gamma+1}{2}t = -zR_{e_2}^-,
\end{gather*}and all other $R^-$ operators are zero.

\medskip

Let $a = x, b = e_1$ in (\ref{r+r-1}): $[R_x^+,R_{e_1}^-] = -\frac{1}{2}R_y^+.$ Applying this relation on $w,$ we get $\gamma = 0.$ Now, one can check that the bimodule $V(\alpha, 0, 0)$ with the $R^-$ actions introduced above is indeed a noncommutative Jordan $D$-bimodule. Indeed, to ensure that we have to check that relations (\ref{r+r-1}), (\ref{r+r-2}) hold for all $a, b \in D.$ Most of them have already been checked in the proof of the Proposition \ref{d0r-}. In fact, the only nontrivial relation that we have not yet checked is (\ref{r+r-1}) with $a = y, b = e_1: [R_y^+,R_{e_1}^-] = 0,$ and it is easy to verify that this relation also holds on this superbimodule. Since $\gamma^2 - 4\alpha\beta - 1 = -1 \neq 0,$ this bimodule is irreducible. We denote this noncommutative Jordan bimodule as $V(\alpha).$ We have proved the following result:
\begin{Th}
Let $M$ be an irreducible finite-dimensional noncommutative Jordan bimodule over $D_{-1}(\frac{1}{2},\frac{1}{2},0)$ and let the base field $\mathbb{F}$ be algebraically closed. Then $M$ is isomorphic to $V(\alpha), \alpha \in \mathbb{F}.$
\end{Th}

\medskip

\subsection{Representations of \texorpdfstring{$K_3(\alpha, \beta, \gamma)$}{K3(α,β,γ)}}
Here, as a consequence of two previous subsections, we obtain a description of irreducible bimodules over nonunital simple noncommutative Jordan superalgebra $K_3(\alpha,\beta,\gamma).$
\medskip

Note that to study representations of $K_3(\alpha, \beta, \gamma)$ it suffices to study unital representations of $D_0(\alpha,\beta,\gamma)$. Indeed, the superalgebra $D_0(\alpha, \beta, \gamma) = D$ is the unital hull of $K_3(\alpha, \beta, \gamma) = K$, and any noncommutative Jordan $K$-bimodule $M$ admits a unital noncommutative Jordan action of $D$ by setting
\[R_{e_2} = id - R_{e_1}, L_{e_2} = id - L_{e_1}.\]
Moreover, one can check that a unital $D$-bimodule $M$ is irreducible if and only if it is irreducible as a bimodule over $K$ with the action induced by embedding. Thus, it suffices to study irreducible unital bimodules over $D$.

Recall that if $\mathbb{F}$ allows square root extraction, then $K_3(\alpha, \beta, \gamma)$ is isomorphic either to $K_3(\lambda,0,0)$ for $\lambda \in \mathbb{F}$ or to $K_3(\frac{1}{2},\frac{1}{2},0).$ Hence, from Theorems \ref{D01-rep}, \ref{dt1/2-rep} we have the following result:
\begin{Th}
Suppose that the base field $\mathbb{F}$ allows square root extraction. Then every finite-dimensional irreducible noncommutative Jordan bimodule over $K_3(\alpha, \beta, \gamma)$ is isomorphic either to the regular bimodule over $K_3(\alpha, \beta, \gamma)$ or its opposite, or the one-dimensional zero bimodule.
\end{Th}

\medskip

\section{Kronecker factorization theorem for \texorpdfstring{$D_t(\alpha, \beta, \gamma)$}{Dt(α,β,γ)}}

A well known theorem, due to Wedderburn, states that if $B$ is an associative algebra and $A$ is a finite dimensional central simple subalgebra of $B$ that contains its unit element, then $B$ is the tensor product of algebras $A$ and $Z$ where $Z$ is the subalgebra of elements of $B$ that commute with every element of $A$. The statements of this type are usually called \textit{Kronecker factorization theorems}. In the paper \cite{Jac_Kron} Jacobson proved the Kronecker factorization theorem for the split Cayley-Dickson algebra and the exceptional simple Albert algebra. In the case of superalgebras, L\'{o}pez-D\'{\i}az and Shestakov proved the Kronecker factorization theorem for simple alternative superalgebras $B(1,2)$ and $B(4,2)$ \cite{LDSA} and for simple Jordan superalgebras $H_3(B(1,2)), H_3(B(4,2))$ obtained from them \cite{LDSJ}.

Clearly, a necessary condition for the Kronecker factorization over a (super)algebra $A$ to hold is that every $A$-module be completely reducible and irreducible summands be isomorphic to $\operatorname{Reg}(A)$ (or $\operatorname{Reg}(A)^{op}$). This is exactly what we proved in the previous section for the superalgberas $D_t(\alpha, \beta, \gamma)$ except for some special values of parameters. Hence, in this section we investigate if the Kronecker factorization holds for these superalgebras. Again, by different methods for each subclass, we obtain the same result: except for some special values of parameters, any noncommutative algebra $U$ that contains $D_t(\alpha,\beta,\gamma)$ as a unital subalgebra is the graded tensor product of $D_t(\alpha,\beta,\gamma)$ and an associative-commutative superalgebra $A.$ As a consequence, we obtain the classification of noncommutative Jordan representations and the Kronecker factorization for simple associative superalgebra $Q(2).$

\subsection{Kronecker factorization theorem for \texorpdfstring{$D_t(\lambda)$}{Dt(λ)}}

We consider first the case $\lambda = 1,$ hoping to apply the mutation later. Let $U$ be a noncommutative Jordan superalgebra that contains $D_t(1) = D$ as a unital subsuperalgebra. Then $U$ can be considered as a unital bimodule over $D.$ From Theorem \ref{Dt1-rep} it follows that $U$ is a direct sum of regular $D$-bimodules and opposite to them:
\[U = \bigoplus_{i \in I} M_i \oplus \bigoplus_{j \in J} \overline{M}_j,\]
where $M_i$ are isomorphic to $\operatorname{Reg}(D)$, and $\overline{M}_j$ are isomorphic to the $\operatorname{Reg}(D)^{op}$. For $a \in D,$ $i \in I$ $(j \in J)$ by $a_i$ $(\overline{a}_j)$ we denote the image of $a$ with respect to the module isomorphism $\operatorname{Reg}(D) \to M_i$ $(\operatorname{Reg}(D)^{op} \to \overline{M}_j).$ From now on by $U = U_0 + U_1 + U_2$ we denote the Peirce decomposition of $U$ with respect to $e_1 \in D.$

\medskip

Consider the space $Z = \{a \in U| [a,D] = 0\}.$ It is easy to see that the commutative center of $D$ is equal to $\mathbb{F}.$ Therefore, the module structure of $U$ implies that $Z = \langle 1_i, i \in I, \overline{1}_j, j \in J \rangle,$ thus, $Z \subset U_0 + U_2.$

\begin{Lem}
\label{Z_subalg}
$Z$ is a subalgebra of $U.$
\end{Lem}
\begin{proof} Let $a, b \in Z, c \in U.$ Then
\[[a \circ b,c] = 2aR_b^+R_c^- = \text{(by (\ref{r+r-1}))} = (-1)^{bc}2aR_c^-R_b^+ = 0.\]
Since $Z \subseteq U_0 + U_2,$ $Z^2$ also lies in $U_0 + U_2$ and $[Z^2,e_1] = [Z^2,e_2] = 0.$ Therefore, we only have to show that $[[a,b],x] = [[a,b],y] = 0.$ This can be showed as follows:
\[[[a,b],x] = 4aR_b^-R_x^- = 4aR_b^-(P_0 + P_2)R_x^- = \text{(by Lemma \ref{r+r-rell})} =\]
\[4aR_b^-(-P_0 + P_2)R_x^+ = \text{(by (\ref{comm_proj}))} = 4a(-P_0 + P_2)R_b^-R_x^+ = \text{(by (\ref{r+r-1}))} =\]
\[(-1)^{b}4a(-P_0 + P_2)R_x^+R_b^- = (-1)^{b}4a(P_0 + P_2)R_x^-R_b^- = (-1)^{b}4aR_x^-R_b^- = 0.\]
Analogously one can show that $[[a,b],y] = 0.$ Hence, $[Z^2,D] = 0$ and $Z$ is a subalgebra of $U$.
\end{proof}

Note that the module structure of $U$ implies that $U_1 = U_1^{[0]} \oplus U_1^{[1]}.$ We will extensively use this property to prove some associativity conditions. In fact, we will show that $U_0 + U_2$ lies in the associative center of $U.$

\begin{Lem}
\label{U0+U2asU1}
$(U_1, U_0 + U_2, U_0 + U_2) = 0, ~(U_0 + U_2, U_1, U_0 + U_2) = 0, ~(U_0 + U_2, U_0 + U_2, U_1) = 0.$
\end{Lem}
\begin{proof} Let $u_0, u_0' \in U_0, a \in U_1^{[0]}, b \in U_1^{[1]}, u_2, u_2' \in U_2.$ Then
\[(a,u_0,u_0') = a(R_{u_0}R_{u_0'} - R_{u_0u_0'}) = \text{(by lemma \ref{psu1id})} = -(-1)^{u_0u_0'}aR_{u_0'}L_{u_0} = \text{(by Lemma \ref{U1^iU_i})} = 0,\]
\[(a,u_2,u_2') = a(R_{u_2}R_{u_2'} - R_{u_2u_2'}) = \text{(by Lemma \ref{psu1id})} = -(-1)^{u_2u_2'}aL_{u_2'}R_{u_2} = \text{(by Lemma \ref{U1^iU_i})} = 0.\]
Analogously, $(b,u_0,u_0') = 0, (b,u_2,u_2') = 0.$

\medskip

Also
\[(a,u_0,u_2) = (au_0)u_2 = \text{(by Lemma \ref{U1^iU_i})} = 0,\]
\[(u_0,u_2,a) = -u_0(u_2a) = \text{(by Lemma \ref{U1^iU_i})} = 0.\]
Analogously, $(b,u_2,u_0) = 0, (u_2,u_0,b) = 0.$

\medskip

Also
\[(u_0,a,u_0') = (u_0a)u_0' \in \text{(by Lemma \ref{U_lambda})} \in U_1^{[0]}u_0' = 0,\]
\[(u_2,a,u_2') = -u_2(au_2') \in \text{(by Lemma \ref{U_lambda})} \in u_2U_1^{[0]} = 0.\]
Analogously, $(u_0,b,u_0') = 0, (u_2,b,u_2') = 0.$

\medskip

By Lemma \ref{U1^iU_1^i} $(u_2,a,u_0) = 0, (u_0,b,u_2) = 0.$ Finally, the arbitrariness of $u_0, u_0', a, b, u_2, u_2'$ and the flexibility relation (\ref{flex_2}) imply the lemma statement.
\end{proof}

\begin{Lem}
\label{U0+U2as}
$U_0 + U_2$ is an associative superalgebra of $U.$
\end{Lem}
\begin{proof} It suffices to show that $U_0$ and $U_2$ are associative. Consider the following identity which is valid in any algebra (\cite[p. 136]{kolca}):
\[(ab,c,d) + (a,b,cd) - a(b,c,d) - (a,b,c)d - (a,bc,d) = 0.\]
Substituting in it $a = x;~ b, c, d \in U_0$ by the previous lemma we get $x(b,c,d) = 0.$ Then the structure of $U$ as a module over $D$ implies that $(b,c,d) = 0$ and $U_0$ is associative. Analogously, substituting $a = y;~ b, c, d \in U_2$ we infer that $U_2$ is associative.
\end{proof}

\begin{Lem}
\label{U1_mult}
$U_1^2 \subseteq U_0 + U_2, (U_1^{[0]})^2 = (U_1^{[1]})^2 = 0,~ U_1^{[0]}U_1^{[1]} \subseteq U_0,~ U_1^{[1]}U_1^{[0]} \subseteq U_2.$
\end{Lem}
\begin{proof} First we prove that the set $K = \{x, y\}$ satisfies the conditions of the Lemma \ref{K_set}. Indeed, the first condition follows automatically from the bimodule structure of $U$ over $D.$ Suppose now that $a \in U_1$ is such that $K \circ a = 0.$ The bimodule structure of $U$ over $D$ implies that
\[a = \sum \alpha_i x_i + \sum \beta_i y_i + \sum \gamma_j \overline{x}_j + \sum \delta_j \overline{y}_j,\]
where $\alpha_i, \beta_i, \gamma_j, \delta_j \in \mathbb{F}, i \in I, j \in J.$ Hence,
\[0 = x \circ a = \sum \beta_i (e_1\phantom{}_i + t e_2\phantom{}_i) + \sum \delta_i (\overline{e_1}_i + t \overline{e_2}_i),\]
therefore, $\beta_i = \delta_j = 0, i \in I, j \in J.$
Analogously, since $y \circ a = 0, \alpha_i = \gamma_j = 0$ for all $i \in I, j \in J.$ Therefore, $a = 0.$ Thus, Lemma \ref{K_set} implies that $U_1^2 \subseteq U_0 + U_2,$ and lemma \ref{U1^iU_1^i} implies the lemma statement.
\end{proof}

\begin{Lem}
\label{U0+U2asU1U1}
$(U_0 + U_2, U_1, U_1) = 0,~ (U_1,U_0 + U_2, U_1) = 0,~ (U_1, U_1, U_0 + U_2) = 0.$
\end{Lem}
\begin{proof} Let $u_0 \in U_0, a, a' \in U_1^{[0]}, b, b' \in U_1^{[1]}, u_2 \in U_2.$ Then
\[(u_0,b,a) = (u_0b)a - u_0(ba) = \text{(by Lemmas \ref{U1^iU_i}, \ref{U1_mult})} = 0,\]
\[(u_2,a,b) = (u_2a)b - u_2(ab) = \text{(by Lemmas \ref{U1^iU_i}, \ref{U1_mult})} = 0.\]
Analogously, $(b,a,u_0) = 0, (a,b,u_2) = 0.$ Now,
\[(u_0,b,b') = (u_0b)b' - u_0(bb') = \text{(by Lemmas \ref{U1^iU_i}, \ref{U1_mult})} = 0,\]
\[(u_2,b,b') = (u_2b)b' - u_2(bb') = \text{(by Lemmas \ref{U_lambda}, \ref{U1_mult})} = 0.\]
Analogously, $(u_0,a,a') = 0, (u_2,a,a') = 0.$ Now,
\[(a,u_0,b) = (au_0)b - a(u_0b) = \text{(by Lemma \ref{U1^iU_i})} = 0,\]
\[(b,u_2,a) = (bu_2)a - b(u_2a) = \text{(by Lemma \ref{U1^iU_i})} = 0,\]
\[(a,u_0,a') = (au_0)a' - a(u_0a') = \text{(by Lemmas \ref{U_lambda}, \ref{U1_mult})} = 0,\]
\[(b,u_0,b') = (bu_0)b' - b(u_0b') = \text{(by Lemmas \ref{U_lambda}, \ref{U1_mult})} = 0.\]
Analogously, $(a,u_2,a') = 0, (b,u_2,b') = 0.$

Finally, arbitrariness of $u_0, a, a' b, b', u_2$ and the flexibility relation (\ref{flex_2}) imply the statement of the lemma.
\end{proof}

\begin{Lem}
$U$ is isomorphic to the graded tensor product of $Z$ and $D.$
\end{Lem}
\begin{proof} Lemmas \ref{U0+U2asU1}, \ref{U0+U2as}, \ref{U0+U2asU1U1} imply that $U_0 + U_2$ lies in the associative center of $U.$ Hence, $Z$ also lies in the associative center of $U.$ Let $a, b \in D, z, z' \in Z.$ Then
\[(za)(z'b) = ((za)z')b = (z(az'))b = (-1)^{az'}(z(z'a))b = (-1)^{az'}((zz')a)b = (-1)^{az'}(zz')(ab).\]
Therefore, $U$ is a homomorphic image of the graded tensor product of $Z$ and $D.$ Since $Z = \langle 1_i, i \in I, \overline{1}_j, j \in J \rangle,$ it is clear that the equality $z_1e_1 + z_2e_2 + z_3x + z_4y = 0$ for $z_1, z_2, z_3, z_4 \in Z$ implies $z_1 = z_2 = z_3 = z_4 = 0.$ Thus, $U \cong Z \otimes D.$

\end{proof}

By Lemma \ref{U0+U2as}, $Z$ is an associative superalgebra. Suppose that $t = -1.$ Then $D$ is isomorphic to an associative superalgebra $M_{1,1}$, and $U$ is also associative as the graded tensor product of two associative superalgebras. If $t \neq -1$, we can specify the structure of $Z$ further:

\begin{Lem}
Suppose that $t \neq -1.$ Then $Z$ is supercommutative.
\end{Lem}
\begin{proof} Let $z_1, z_2, z_3 \in Z.$ Then by associativity of $Z$ we have
\[(z_1 \otimes x, z_2 \otimes y, z_3 \otimes x) = (-1)^{z_2}(z_1z_2z_3)\otimes(x,y,x).\]
The flexibility relation (\ref{flex_2}) implies that
\[(z_1 \otimes x, z_2 \otimes y, z_3 \otimes x) = -(-1)^{z_1 \otimes x, z_2 \otimes y, z_3 \otimes x}(z_3 \otimes x, z_2 \otimes y, z_1 \otimes x) = (-1)^{z_1,z_2,z_3}(z_3 \otimes x, z_2 \otimes y, z_1 \otimes x).\]
Hence,
\[0 = (z_1z_2z_3 - (-1)^{z_1,z_2,z_3}z_3z_2z_1) \otimes (x,y,x) = 2(1+t)(z_1z_2z_3 - (-1)^{z_1,z_2,z_3}z_3z_2z_1) \otimes x.\]
Therefore, $z_1z_2z_3 - (-1)^{z_1,z_2,z_3}z_3z_2z_1 = 0.$ Taking $z_3 = 1,$ we get $z_1z_2 = (-1)^{z_1z_2}z_2z_1.$
\end{proof}

\medskip

Consider now the general situation, that is, let $U$ be a noncommutative Jordan superalgebra that contains $D_t(\lambda)$ as a unital subsuperalgebra. Suppose that $\lambda \neq 1/2.$ Therefore, $U' = U^{(\mu)}$ contains $D_t(1)$ as a unital subsuperalgebra, where $\mu \in \mathbb{F}$ is such that $\lambda \odot \mu = 1.$ By what was proved before, $U' = Z \otimes D_t(1)$ for an associative superalgebra $Z,$ and $U = U'^{(\lambda)} = (Z \otimes D_t(1))^{(\lambda)}.$ Suppose that $Z$ is supercommutative and let $z, z' \in Z, a, b \in D_t(1).$ Then
\[(z\otimes a)\cdot_\lambda(z'\otimes b) = \lambda (z \otimes a) (z' \otimes b) + (-1)^{(z + a)(z' + b)}(1-\lambda)(z'\otimes b)(z\otimes a) = \]
\[(-1)^{az'}\lambda(zz') \otimes (ab) + (-1)^{az' + ab}(1-\lambda)(zz')\otimes(ba) = (-1)^{az'}(zz')\otimes(a\cdot_\lambda b).\]
Therefore, if $Z$ is supercommutative (which holds, for example, when $t \neq -1$), $U$ is isomorphic to $Z \otimes D_t(1)^{(\lambda)} = Z \otimes D_t(\lambda).$ Now we can state our main result:

\begin{Th}
\label{D_Kron}
Let $U$ be a noncommutative Jordan superalgebra that contains $D_t(\lambda)$ as a unital subsuperalgebra, $t \neq 0, 1, \lambda \neq 1/2.$ Then$:$\\
$1)$ If $t \neq -1,$ then $U \cong Z \otimes D_t(\lambda)$, where $Z$ is an associative-commutative superalgebra$;$\\
$2)$ If $t = -1$, then $U \cong (Z \otimes M_{1,1})^{(\lambda)},$ where $Z$ is an associative superalgebra. Particularly, any noncommutative Jordan superalgebra containing $D_{-1}(1) = M_{1,1}$ as a unital subsuperalgebra is associative.
\end{Th}

\medskip

\begin{Remark}Note that the condition $\lambda \neq 1/2$ and $t \neq 0,1$ is necessary for the theorem. Indeed, when $\lambda = 1/2, D_t(1/2)$ is just the Jordan superalgebra $D_t$ which has Jordan bimodules non-isomorphic neither to $\operatorname{Reg}(D_t)$ nor to $\operatorname{Reg}(D_t)^{op}$ \cite{Trush}, \cite{Trushmod}. If $t = 0,$ then it is easy to see that $\operatorname{Reg}(D_0(\alpha, \beta, \gamma))$ has a 3-dimensional submodule generated by $e_1, x, y.$ An example of noncommutative Jordan bimodule over $D_1(\lambda)$ which is not isomorphic to $\operatorname{Reg}(D_1(\lambda))$ or $\operatorname{Reg}(D_1(\lambda))^{op}$ will be given further in the paper.
\end{Remark}

\medskip

\subsection{Kronecker factorization theorem for \texorpdfstring{$D_t(\frac{1}{2},0,0)$}{Dt(1/2,0,0)}}
Let $(U,\circ,\{\cdot,\cdot\})$ be a noncommutative Jordan superalgebra containing $D_t(\frac{1}{2},\frac{1}{2},0) = D$ as a unital superalgebra, $t \neq -1, 0, 1.$ Consider $U$ as a unital superbimodule over $D.$ As before, Theorem \ref{dt1/2-rep} implies that
\begin{equation}
\label{U_decomp}
U = \bigoplus_{i \in I} M_i \oplus \bigoplus_{j \in J} \overline{M}_j,
\end{equation}
 where $M_i \cong \operatorname{Reg}(D)$, and $\overline{M}_j \cong \operatorname{Reg}(D)^{op}$. For $a \in D,$ $i \in I$ $(j \in J)$ by $a_i$ $(\overline{a}_j)$ we denote the image of $a$ with respect to the module isomorphism $\operatorname{Reg}(D) \to M_i$ $(\operatorname{Reg}(D)^{op} \to \overline{M}_j).$ By $U = U_0 + U_1 + U_2$ we denote the Peirce decomposition of $U$ with respect to $e_1 \in D.$

As in the previous section, let $Z = \langle 1_i, i \in I, \overline{1}_j, j \in J \rangle.$ In this section we prove that $Z$ is a supercommutative subalgebra of $U$ and $U \cong Z \otimes D.$

\begin{Lem}
\label{Z_subalg_com}
$Z$ is a commutative subalgebra of $U$.
\end{Lem}
\begin{proof} The module structure of $U$ implies that $Z = \{a \in U| [a,D] = 0\} \cap (U_0 + U_2).$ It is obvious that $Z^2 \subseteq U_0 + U_2,$ thus, to prove that $Z$ is a subalgebra it suffices to show that $[Z^2,D] = 0.$ If $a, b \in Z, c \in D$, then
\[[a \circ b,c] = 2aR_b^+R_c^- = \text{(by (\ref{r+r-1}))} = (-1)^{bc}2aR_c^-R_b^+ = 0.\]
To show that $[Z,Z] = 0$ it suffices to prove that $U_0 + U_2$ is a commutative subalgebra of $U.$ In the same way as in Lemma \ref{U1_mult} we can prove that $U_1U_1 \subseteq U_0 + U_2.$ Then, since
\[[e_1,y_i] = 0, [e_1,\overline{y}_j] = 0, [e_1,x_i] = y_i, [e_1,\overline{x}_j] = \overline{y}_j\]
for all $i \in I, j \in J$, Lemmas \ref{r+r-rell} and \ref{r+r-rel} imply that
\[R_{y_i}^- = 0,~ R_{\overline{y}_j}^- = 0,~ P_1R_{x_i}^- = P_1R_{y_i}^+(P_0 - P_2),~ P_1R_{\overline{x}_j}^- = P_1R_{\overline{y}_j}^+(P_0 - P_2).\]

Now,
\[y_i \circ y_j = y_iR_{y_j}^+(P_0 + P_2) = y_iR_{x_j}^-(P_0 - P_2) = \frac{1}{2}[y_i,x_j](P_0 - P_2) = x_jR_{y_i}^-(P_0 - P_2) = 0.\]
Analogously we can prove that $y_i \circ \overline{y}_j = 0, \overline{y}_i \circ \overline{y}_j = 0.$

Consider now the relation (\ref{r+r-2}) with $a = x, b = y_i: R_x^-R_{y_i}^+ = R_{(e_1 + te_2)_i}^-$. Apply the resulting relation on the element $e_{1_j}:$
\[\frac{1}{2}[e_{1_j},e_{1_i}] = e_{1_j}R_x^-R_{y_i}^+ = \frac{1}{2}y_j \circ y_i = 0.\]
Analogously (using the relation (\ref{r+r-2}) with $a = x, b = \overline{y}_i$) we can show that
\[[e_{1_i},\overline{e}_{1_j}] = 0,~ [\overline{e}_{1\phantom{}_i},\overline{e}_{1\phantom{}_j}] = 0,\]
therefore, $U_2$ is supercommutative. Since $t \neq 0$, we can apply the same identities to elements $e_{2_i}, \overline{e}_{2_j}, i \in I, j \in J$ to prove that $U_0$ is supercommutative. Therefore, $Z$ is a supercommutative subalgebra of $U$.
\end{proof}
\medskip

With this lemma we may consider only the symmetrized Jordan superalgebra:

\begin{Lem}
Suppose that $U^{(+)}$ is the graded tensor product of Jordan superalgebras $Z$ and $D^{(+)} = D_t.$ Then $U$ as a noncommutative Jordan superalgebra is the graded tensor product of $Z$ and $D.$
\end{Lem}
\begin{proof} In the proof of the previous lemma we noted that $U_0 + U_2$ is a commutative subsuperalgebra of $U$, and that $R_{y_i}^- = R_{\overline{y}_j}^- = 0$ for all $i \in I, j \in J.$ Hence, every nonzero commutator in $U$ is a sum of the commutators of the form $[a,x_i]$, $[b,\overline{x}_j]$, where $a, b \in U_0 + U_2 + \langle x_i, \overline{x}_j , i \in I, j \in J \rangle.$ Consider, for example, the commutator of elements $x_i$ and $x_j:$
\[[(1_i \otimes x),(1_j \otimes x)] = [x_i,x_j] = 2x_iR_{x_j}^- = 2x_iR_{y_j}^+(P_0-P_2) = \]
\[2(1_i \otimes x) \circ (1_j \otimes y)(P_0-P_2) = 2(1_i \circ 1_j) \otimes (x \circ y)(P_0-P_2) = \text{(by lemma \ref{Z_subalg_com})} = \]
\[2 1_i1_j \otimes (-e_1 + te_2) = 1_i1_j \otimes [x,x].\]
Thus,
\[(1_i \otimes x)(1_j \otimes x) = (1_i \otimes x)\circ (1_j \otimes x) + \frac{1}{2}[(1_i \otimes x),(1_j \otimes x)] = 1_i1_j \otimes (x \circ x + \frac{1}{2}[x,x]) = 1_i1_j \otimes x^2.\]

Consider also the commutator of elements $\overline{e}_{1\phantom{}_i}$ and $x_j:$
\[[(\overline{1}_i \otimes e_1),(1_j \otimes x)] =  [\overline{e}_{1\phantom{}_i},x_j] = 2\overline{e}_{1\phantom{}_i}R_{x_j}^- = 2\overline{e}_{1\phantom{}_i}R_{y_j}^+ = \]
\[2(\overline{1}_i \otimes e_1)\circ(1_j \otimes y) = 2 (\overline{1}_i \circ 1_j) \otimes (e_1 \circ y) = \text{(by Lemma \ref{Z_subalg_com})} = \]
\[(\overline{1}_i1_j) \otimes y = (\overline{1}_i1_j) \otimes [e_1,x].\]
Hence,
\[(\overline{1}_i \otimes e_1)(1_j \otimes x) = (\overline{1}_i \otimes e_1)\circ(1_j \otimes x) + \frac{1}{2}[(\overline{1}_i \otimes e_1),(1_j \otimes x)] = \overline{1}_i1_j \otimes (e_1 \circ x + \frac{1}{2}[e_1,x]) = \overline{1}_i1_j \otimes e_1x.\]

Other cases of $a, b$ can be dealt with completely analogously. For other pairs of elements the product in $U$ coincides with the product in $U^{(+)}.$ Therefore, $U \cong Z \otimes D$ as a noncommutative Jordan superalgebra.
\end{proof}

Now it is only left for us to show that $U^{(+)} \cong Z \otimes D_t.$ We prove this analogously to the paper \cite{MZKr}. From now on until the end of the section we will be working with a Jordan superalgebra $U^{(+)},$ thus, for convenience we will denote it as $U$, write Jordan product in $U$ as juxtaposition and denote operators $R_x^+, x \in U,$ as $R_x$. Henceforth we may suppose that $t \neq -3.$ Indeed, if char $\mathbb{F} = 3$, then $t = 0$, which we have already excluded, and if char $\mathbb{F} \neq 3,$ then $D_{-3} \cong D_{-\frac{1}{3}}$. First of all we need some preliminary data about derivations:

\medskip

\begin{Def} Let $A$ be a superalgebra, and $M$ be a superbimodule over $A$. A mapping $d: A \to M$ is called a \textit{derivation from $A$ to $M$} if 
\[(ab)d = a(bd) + (-1)^{bd}(ad)b\]
for all $a, b \in A.$ The space of derivations from $A$ to $M$ is denoted by $\operatorname{Der}(A,M).$ If $A$ is considered as a module over itself then an element $d \in \operatorname{Der}(A,A)$ is called a derivation of $A$. The space $\operatorname{Der}(A) = \operatorname{Der}(A,A)$ with the Lie superalgebra structure is called \textit{the algebra of derivations of $A$}.
\end{Def}

Let $J$ be a Jordan superalgebra. For elements $a, b \in J$ the operator $D(a,b) = [R_a,R_b]$ is a derivation of $J.$ Derivations of the form $\sum D(a_i,b_i),~ a_i, b_i \in J$ are called \textit{inner}.
We need to know some facts about the algebra of derivations of $D_t$.

\begin{Lem}
\label{Dt_der}
The superalgebra of derivations of $D_t, t \neq -1$ is a simple 5-dimensional Lie superalgebra, and $D_t/\mathbb{F}$ is an irreducible $\operatorname{Der}(D_t)$-module. Also every derivation of $D_t$ is inner.
\end{Lem}
\begin{proof} The computation of this algebra is rather straightforward, so we omit it and only present the base of $\operatorname{Der}(D_t)$ (basis elements of $D_t$ on which the derivations below are not defined map to zero):
\begin{gather*}
e: x \mapsto y, ~ f: y \mapsto x, ~ h: \begin{cases} x \mapsto x,\\ y \mapsto -y \end{cases}\\
a: \begin{cases} e_1 \mapsto x, \\ e_2 \mapsto -x,\\ x \mapsto 0,\\ y \mapsto 2(e_1 - te_2) \end{cases} ~ b: \begin{cases} e_1 \mapsto y,\\ e_2 \mapsto -y,\\ x \mapsto 2(-e_1 + te_2),\\y \mapsto 0 \end{cases}.
\end{gather*}

\medskip

Let $M$ be a $\operatorname{Der}(D_t)$-submodule of $D_t$ containing 1. Note that $e,f,h$ span the simple Lie algebra $\mathfrak{sl}_2$ that acts irreducibly on the odd space $\langle x, y \rangle.$ Therefore, if $M$ contains an odd element of $D_t$, it contains all $(D_t)_{\overline{1}}$. Hence (by acting by $a$), it contains an element $e_1 - te_2$. Since $t \neq - 1$, $M$ is equal to the whole $D_t$. If $M$ contains an even element $\neq 1$, it contains (by acting by $a, b$) elements $x$ and $y$ and again is equal to the whole $D_t.$ Hence, $D_t/\mathbb{F}$ is an irreducible $\operatorname{Der}(D_t)$-supermodule.

\medskip

The multiplication table of $\operatorname{Der}(D_t)$ is the following:
\begin{gather*}
[e,f] = h,~ [h,f] = -2f,~ [h,e] = 2e,\\
[a,a] = 4(1+t)f,~ [a,b] = -2(1+t)h,~ [b,b] = -4(1+t)e,\\
[e,a] = -b,~ [f,a] = 0,~ [h,a] = -a,\\
[e,b] = 0,~ [f,b] = -a,~ [h,b] = b.
\end{gather*}
Using this table it is quite easy to see that $\operatorname{Der}(D_t)$ is a simple superalgebra.

One can also check that
\begin{gather*}
e = \frac{2}{1+t}[R_y,R_y],~ f = \frac{-2}{1+t}[R_x,R_x],~ h = \frac{2}{1+t}[R_x,R_y];\\
a = 4[R_{e_1},R_x],~ b = 4[R_{e_1},R_y].
\end{gather*}

Therefore, all derivations of $D_t$ are inner.
\end{proof}

Note that the decomposition (\ref{U_decomp}) continues to hold for the symmetrized superalgebra: $M_i, i \in I$ as a Jordan bimodule is isomorphic to $\operatorname{Reg}(D_t)$, and $\overline{M}_j, j \in J$ as a Jordan bimodule is isomorphic to $\operatorname{Reg}(D_t)^{op}.$ From this and the previous lemma it follows that for any $D_t$-submodule $M \subseteq U$ the Lie superalgebra $\operatorname{Der}(D_t)$ acts on $M$ and on $\operatorname{Der}(D_t,M).$ Note also that $\operatorname{Der}(D_t,D_t^{op})$ as a module over $\operatorname{Der}(D_t)$ is isomorphic to the opposite of the regular bimodule, thus is also irreducible.

\medskip

Lemma \ref{Z_subalg_com} implies that $Z$ is a subalgebra of $U^{(+)}.$ Note also that for $a, b \in D_t, i \in I, j \in J$ we have
\begin{equation}
\label{ZJ_as}
(a1_i)b = (ab)1_i,~ b(a1_i) = (ba)1_i,
\end{equation}
\begin{equation}
\label{ZopJ_as}
(a\overline{1}_j)b = (-1)^b(ab)\overline{1}_j,~ b(a\overline{1}_j) = (ba)\overline{1}_j.
\end{equation}
From this equations it follows easily that $Z\operatorname{Der}(D_t) = 0.$

\medskip

We will also need the Jordan identity in the element form:
\[((ab)c)d + (-1)^{b,c,d}((ad)c)b + (-1)^{cd}a((bd)c) = (ab)(cd) + (-1)^{bc}(ac)(bd) + (-1)^{(b+c)d}(ad)(bc) =\]
\[=a((bc)d) + (-1)^{cd}((ab)d)c + (-1)^{b(c+d)}((ac)d)b.\]

Recall that we aim to prove that for $a, b \in D_t, z_1, z_2 \in Z$ we have $(z_1a)(z_2b) = (-1)^{z_2a}(z_1z_2)(ab).$

\begin{Lem}
\label{Z_tp_1}
For $a \in D_t, z_1, z_2 \in Z$ we have $(az_1)z_2 = a(z_1z_2).$
\end{Lem}
\begin{proof} We show that the mapping $d: D_t \to U$ defined by $a \mapsto (az_1)z_2 - a(z_1z_2),$ where $z_1, z_2 \in Z$, is a derivation.
That is, we show that for $a, b \in J$ we have
\[((ab)z_1)z_2 - (ab)(z_1z_2) = a((bz_1)z_2 - b(z_1z_2)) + (-1)^{b(z_1 + z_2)}((az_1)z_2 - a(z_1z_2))b.\]
From relations (\ref{ZJ_as}), (\ref{ZopJ_as}) it follows that $(ab)(z_1z_2) = (-1)^{b(z_1z_2)}(a(z_1z_2))b$. Thus, we need to prove that
\begin{equation}
\label{Z_as_der}
((ab)z_1)z_2 = a((bz_1)z_2 - b(z_1z_2)) + (-1)^{b(z_1 + z_2)}((az_1)z_2)b.
\end{equation}

By the Jordan identity we have
\[(z_2(ab))z_1 + (-1)^{bz_1}(z_2(az_1))b + (-1)^{a(b+z_1)}(z_2(bz_1))a =\]
 \[((z_2a)b)z_1 + (-1)^{a,b,z_1}((z_2z_1)b)a + (-1)^{bz_2}z_2((az_1)b).\]

Since $(z_2(ab))z_1 = ((z_2a)b)z_2$, we have
\[(-1)^{bz_1}(z_2(az_1))b + (-1)^{a(b+z_1)}(z_2(bz_1))a = (-1)^{a,b,z_1}((z_2z_1)b)a + (-1)^{bz_1}z_2((az_1)b),\]
which together with supercommutativity implies (\ref{Z_as_der}). Since $Z\operatorname{Der}(D_t) = 0,$ the derivation $d$ commutes with $\operatorname{Der}(D_t)$. Thus, the compositions of $d$ with the projections $U \to M_i$, $U \to \overline{M}_j$ belong to $\operatorname{Der}(D_t,M_i)$ and $\operatorname{Der}(D_t,\overline{M}_j)$ respectively and also commute with $\operatorname{Der}(D_t).$ Since the action of $\operatorname{Der}(D_t)$ on $\operatorname{Der}(D_t)$ and $\operatorname{Der}(D_t,D_t^{op})$ has only zero constants, we conclude that $d = 0.$
\end{proof}

\begin{Lem}
\label{Z_tp_2}
For $a, b \in D_t, z_1, z_2 \in Z$ we have $(z_1a)(z_2b)= (-1)^{z_2a}(z_1z_2)(ab).$
\end{Lem}
\begin{proof} For fixed elements $z_1, z_2 \in Z, a \in D_t$ consider the mapping
\[d_a:J \to U,~ b\mapsto  (bz_1)(z_2a) -(-1)^{a(z_1+z_2)}(ba)(z_1z_2).\]

We prove that $d_a$ is a derivation. Indeed, by the Jordan identity for $b, b' \in D_t$ we have
\[((bb')z_1)(z_2a) + (-1)^{b'(z_1+z_2+a) + z_1(z_2+a)}((b(z_2a)z_1)b' + (-1)^{z_1(z_2+a)}b(b'(z_2a))z_1) = \]
\[(-1)^{b'(z_1+z_2+a)}((bz_1)(z_2a))b' + b((b'z_1)(z_2a)) + (-1)^{z_1(z_2+a)}((bb')(z_2a))z_1.\]
By previous lemma and relations (\ref{ZJ_as}), (\ref{ZopJ_as}) we have
\[(-1)^{z_1(z_2+a)}((bb')(z_2a))z_1 = (-1)^{a(z_2+z_1)}((bb')a)(z_1z_2),\]
\[(-1)^{z_1(z_2+a)}(b(z_2a))z_1 = (-1)^{a(z_2+z_1)}(ba)(z_1z_2) = (-1)^{z_1(z_2+a)}(b'(z_2a))z_1 = (-1)^{a(z_2+z_1)}(b'a)(z_1z_2).\] Therefore,
\[(bb')d_a = ((bb')z_1)(z_2a) - (-1)^{a(z_2+z_1)}((bb')a)(z_1z_2) = \]
\[b((b'z_1)(z_2a) - (-1)^{a(z_2+z_1)}(b'a)(z_1z_2)) + (-1)^{b'(a+z_1+z_2)}((bz_1)(z_2a) - (-1)^{a(z_2+z_1)}(ba)(z_1z_2))b' =\]
\[b\cdot b'd_a + (-1)^{b'D_a}bd_a \cdot b,\]
and $d_a$ is a derivation. One can easily check that
\[d_{aD{x,y}} = [d_a,D_{x,y}] \text{ for } a, x, y \in D_t,\]
that is, the map $a \mapsto d_a$ is a $\operatorname{Der}(D_t)$-module homomorphism from $D_t$ to $\operatorname{Der}(D_t,U)$. By previous lemma, $\mathbb{F}$ lies in the kernel of this homomorphism, therefore, there is a homomorphism of an irreducible $\operatorname{Der}(D_t)$-module $D_t/\mathbb{F}$ into $\operatorname{Der}(D_t,U).$ If this homomorphism is not zero then one of its compositions with projections to submodules $\operatorname{Der}(D_t,M_i), \operatorname{Der}(D_t,\overline{M}_j)$ is not zero. Hence, $D_t/\mathbb{F}$ is contained in $\operatorname{Der}(D_t,D_t)$ or $\operatorname{Der}(D_t,D_t^{op})$, which are also irreducible $\operatorname{Der}(D_t,D_t)$-bimodules. Since $\dim \operatorname{Der}(D_t,D_t) = \dim \operatorname{Der}(D_t,D_t^{op}) = 5$ and $\dim D_t/\mathbb{F} = 3$, we have obtained a contradiction. Therefore, $D_a = 0$ for all $a \in D_t.$
\end{proof}

\begin{Lem}
\label{Z_as}
$Z$ is an associative superalgebra.
\end{Lem}
\begin{proof} Consider the Jordan identity for $a = z_1\otimes x, b = z_2 \otimes y, c = z_3 \otimes e_1, d = 1 \otimes x, z_1, z_2, z_3 \in Z:$
\[(-1)^{z_2}\frac{1}{2}(z_1z_2)z_3\otimes x - (-1)^{z_2}\frac{1}{2}z_1(z_2z_3)\otimes x = \]
\[= (-1)^{z_2}\frac{1+t}{4}(z_1z_2)z_3 \otimes x - (-1)^{z_2 + z_2z_3}\frac{1+t}{4}(z_1z_3)z_2 \otimes x.\]
Therefore, we have
\[(z_1,z_2,z_3) = (-1)^{z_1z_2}\frac{1+t}{2}(z_2,z_1,z_3) = \frac{(1+t)^2}{4}(z_1,z_2,z_3).\]
Since we have excluded the cases $t = 1, -3$, $Z$ is associative.
\end{proof}

We have now proved the main result of the section:

\begin{Th}
Let $U$ be a noncommutative Jordan superalgebra containing $D_t(\frac{1}{2},\frac{1}{2},0) = D, t \neq -1, 0, 1$ as a unital subsuperalgebra. Then $U \cong Z \otimes D,$ where $Z$ is an associative-supercommutative superalgebra.
\end{Th}

\medskip

\section{Representations of \texorpdfstring{$Q(1)$ and $Q(2)$}{Q(1) and Q(2)}}
The associative superalgebra $Q(n)$ is a subalgebra of the full matrix superalgebra $M_{n,n}(\mathbb{F})$ with the following grading:

\[Q(n)_{\bar{0}} = \Bigg\{ \begin{pmatrix} A & 0 \\ 0 & A \end{pmatrix}, A \in M_n(\mathbb{F}) \Bigg\},~
Q(n)_{\bar{1}} = \Bigg\{ \begin{pmatrix} 0 & B \\ B & 0 \end{pmatrix}, B \in M_n(\mathbb{F}) \Bigg\}.\]
We can consider $Q(n)$ as the double of the $n \times n$ matrix algebra:
\[Q(n) = M_n(\mathbb{F}) \oplus \overline{M_n(\mathbb{F})},\]
where $\overline{M_n(\mathbb{F})}$ is an isomorphic copy of $M_n(\mathbb{F})$ as a vector space. The grading on $Q(n)$ is then
\[Q(n)_{\bar{0}} = M_n(\mathbb{F}),~ Q(n)_{\bar{1}} = \overline{M_n(\mathbb{F})}.\]
The multiplication in $Q(n)$ is defined in the following way:
\[a\cdot b = ab,~ \bar{a} \cdot b = a \cdot \bar{b} = \overline{ab},~ \bar{a} \cdot \bar{b} = ab,\]
where $a, b \in M_n(\mathbb{F})$, and $ab$ is the usual matrix product.
It is widely known that $Q(n)$ is a simple associative superalgebra for all natural $n$.

It is easy to see that the degree of $Q(n)$ is exactly $n$. Thus, noncommutative Jordan representations of $Q(n), n \geq 3$, are described in theorem \ref{Repdeggeq3}.

In this section we describe noncommutative Jordan representations of $Q(1)$ and $Q(2).$

\subsection{Representations of \texorpdfstring{$Q(1)$}{Q(1)}}
The superalgebra $Q(1)$ has a basis $1, \overline{1},$ where $1$ is the unit of the superalgebra, and $\overline{1}^2 = 1.$

Alternative representations of $Q(1)$ were studied by Pisarenko in \cite{Pis}. In particular, he described all irreducible alternative representations of $Q(1)$ and found a series of indecomposable alternative superbimodules over this algebra.

\medskip

We note that in fact all unital representations of $Q(1)$ are alternative.
Indeed, a superalgebra $A$ is alternative if and only if it satisfies the following operator relations:
\[L_{x\circ y} = L_x \circ L_y,~ R_{x \circ y} = R_x \circ R_y, x, y \in A.\]
Let $M$ be a unital bimodule over $Q(1)$. Then it is easy to see that the above relations trivially hold in the split zero extension $Q(1)\oplus M$. Thus, we have proved
\begin{proposition}
Any unital representation of $Q(1)$ is alternative.
\end{proposition}

\subsection{Representations of \texorpdfstring{$Q(2)$}{Q(2)}}
$Q(2)$ is an 8-dimensional simple associative superalgebra. Regarding to an idempotent $e_{11}$, we have the following Peirce decomposition of $Q(2) = U:$
\[U_0 = \langle e_{22}, \overline{e_{22}} \rangle,~U_1 = \langle e_{12}, e_{21}, \overline{e_{12}}, \overline{e_{21}} \rangle,~U_2 = \langle e_{11}, \overline{e_{11}}\rangle.\]

Alternative representations of $Q(2)$ were studied by Pisarenko \cite{Pis}. Particularly, he described irreducible unital bimodules over $Q(2)$ (all of them turned out to be associative and isomorphic either to $\operatorname{Reg}(Q(2))$ or $\operatorname{Reg}(Q(2))^{op}$) and proved that every bimodule over $Q(2)$ is completely reducible.

\medskip

In the paper \cite{ps2} it was proved that a noncommutative Jordan superalgebra $U$ such that $U^{(+)} \cong Q(2)^{(+)}$ is necessarily its mutation: $U \cong Q(2)^{(\lambda)}, \lambda \in \mathbb{F}.$ So, using the module mutation, it suffices to study the representations of $Q(2)$ and $Q(2)^{(+)}.$ Description of noncommutative Jordan representations of $Q(2)$ is a consequence of the results of previous section:
\begin{Th}
Any unital noncommutative Jordan representation of $Q(2)$ is associative.
\end{Th}
\begin{proof} Let $M$ be a noncommutative Jordan bimodule over $Q(2)$ and $E$ be the corresponding split null extension. Note that $Q(2)$ contains a subalgebra $D = \langle e_{11}, \overline{e_{12}}, \overline{e_{21}}, e_{22} \rangle,$ which contains the unit of $Q(2)$ and is isomorphic to $D_{-1}(1) \cong M_{1,1}$. Therefore, $E$ contains $D$ as a unital subalgebra, so by Theorem \ref{D_Kron}, $E$ is associative and $M$ is an associative bimodule over $U.$
\end{proof}

Moreover, as a consequence of Theorem \ref{D_Kron}, we can prove the Kronecker factorization theorem for $Q(2):$
\begin{Th}
Let $U$ be a noncommutative Jordan superalgebra that contains $Q(2)$ as a unital superalgebra. Then $U$ is associative and $U \cong Z \otimes Q(2)$, for an associative superalgebra $Z$.
\end{Th}
\begin{proof} Since $Q(2)$ contains a noncommutative Jordan subsuperalgebra which is isomorphic to $D_{-1}(1)$, by Theorem \ref{D_Kron}, $U$ is an associative algebra. From Pisarenko's classification of unital superbimodules over $Q(2)$ it follows that
\[U = \bigoplus_{i \in I} M_i \oplus \bigoplus_{j \in J} \overline{M}_j,\]
where $M_i$ are isomorphic to $\operatorname{Reg}(Q(2))$, and $\overline{M}_j$ are isomorphic to the $\operatorname{Reg}(Q(2))^{op}$. For $a \in Q(2),$ $i \in I$ $(j \in J)$ by $a_i$ $(\overline{a}_j)$ we denote the image of $a$ with respect to the module isomorphism $\operatorname{Reg}(Q(2)) \to M_i$ $(\operatorname{Reg}(Q(2))^{op} \to \overline{M}_j).$

Consider the set $Z = \langle 1_i, i \in I, \overline{1}_j, j \in J \rangle.$ Since the commutative center of $Q(2)$ is equal to $\mathbb{F},$ it is clear that $Z = \{a \in U: [a,Q(2)] = 0\}.$ Since $U$ is associative, $Z$ is a subalgebra of $U.$ It is clear that
\[(za)(z'b) = (-1)^{az'}(zz')(ab) \text{ for } z, z' \in Z,~ a, b \in Q(2)\]
and the definition of $Z$ implies that every element of $U$ can be uniquely expressed as a sum $\sum_{b \in B}z_bb$, where $B$ is a basis of $Q(2)$, $z_b \in Z, b \in B$. Hence, $U \cong Z \otimes B.$
\end{proof}

\section{Noncommutative Jordan representations of simple Jordan superalgebras}

In this section we study unital noncommutative Jordan representations of simple Jordan (super)algebras. For the (super)algebra $J(V,f)$ of nondegenerate vector form we find a class of representations that are not Jordan, and prove that an irreducible noncommutative Jordan representation of such (super)algebra is Jordan or belongs to the described class. For low-dimensional superalgebras $D_t, K_3, P(2), Q(2)^{(+)}, K_{10}$ and $K_9$ we develop a unified approach which allows us to prove that any (not just irreducible) noncommutative Jordan representation over them is Jordan. For the Kac superalgebra $K_{10}$ in the case of $\mathbb{F}$ algebraically closed and characteristic zero it means that any unital $K_{10}$-module is completely reducible with irreducible summands regular module and its opposite \cite{Sht}, so we check that the Kronecker factorization for $K_{10}$ holds in the class of noncommutative Jordan superalgebras. We finish the section with open questions and plans for further research.

\medskip

First of all, we prove some technical statements that we will need later. Clearly, in our setting (studying noncommutative Jordan representations of Jordan algebras) it is easier to use the approach given by the Definitions \ref{noncomm_jord_2mult}, \ref{2mult_rep}.

\begin{Lem}
\label{r+r-comm}
Let $J$ be a Jordan superalgebra with an even idempotent $e$ and $M$ a noncommutative Jordan bimodule over $J$. Then the following statements hold:

$1)$ Let $x, y \in J_1.$ Then $R_x^+R_y^- = R_x^-R_y^+ = 0;$

$2)$ $M_1$ is closed under the operators of the form $R_a^{\epsilon_1}R_b^{\epsilon_2},$ where $a, b \in J_1, \epsilon_1, \epsilon_2 \in \{+,-\};$

$3)$ $R_{J_1 \circ J_1}^- = 0;$

$4)$ Let $a \in J_0 + J_2, x \in J_1.$ Then $P_1R_a^-R_x^+ = 0, P_1R_{a \circ x}^- = (-1)^{ax}P_1R_x^-R_a^+;$

$5)$ $M' = M_1R_J^-$ is a $J$-submodule of $M,$ and $M' \subseteq M_1.$
\end{Lem}

\begin{proof} $1)$ From Lemma \ref{r+r-rel} it follows that $(P_0 + P_2)R_x^-R_y^+  = 0$ and $P_1R_x^+R_y^- = 0.$

Since $J$ is commutative, (\ref{r+r-1}) simplifies to
\begin{equation}
\label{r+r-1_comm}
[R_a^{+},R_b^{-}] = 0.
\end{equation}
Therefore,
\[(P_0 + P_2)R_x^+R_y^- = (-1)^{xy}(P_0 + P_2)R_y^-R_x^+ = 0,~ P_1R_x^-R_y^+ = (-1)^{xy}P_1R_y^+R_x^- = 0.\]

$2)$ Easily follows from the previous point, $1)$ of Lemma \ref{r+r-rel} and Peirce relations.

$3)$ Easily follows from $1)$ and (\ref{r+r-2}).

$4)$ From the relation (\ref{r+r-2}) it follows that
\[P_1(R_{a \circ x}^- - (-1)^{ax}R_x^-R_a^+) = P_1R_a^-R_x^+.\]
Using $1)$ of Lemma \ref{r+r-rel} and Peirce relations, it is easy to see that the image of the left part lies in $M_1$, and the image of the right part lies in $M_0 + M_2$, hence the statement follows.

$5)$ We have to check that $M'$ is closed under all operators $R_x^+, R_x^-, x \in J.$ Let $x, y \in J_1, z, t \in J_0 + J_2.$ Then, using points $1), 4),$ (\ref{r+r-1_comm}) and $1)$ of Lemma \ref{r+r-rel} it is easy to see that
\[M_1R_x^-R_y^+ = 0,~ M_1R_x^-R_y^- \subseteq M',~ M_1R_x^-R_z^+ = M_1R_z^+R_x^- \subseteq M',~ M_1R_x^-R_z^- \subseteq M',\]
\[M_1R_z^-R_x^+ = 0,~ M_1R_z^-R_x^- \subseteq M',~ M_1R_z^-R_t^+ = M_1R_t^+R_z^- \subseteq M',~ M_1R_z^-R_t^- \subseteq M',\]
hence, $M'$ is a $J$-submodule of $M,$ and from $1)$ of Lemma \ref{r+r-rel} it follows that $M' \subseteq M_1.$
\end{proof}

Now we are ready to describe noncommutative Jordan bimodules over simple Jordan superalgebras.

\subsection{Representations of the superalgebra \texorpdfstring{$J(V,f)$}{J(V,f)}.}
Jordan representations of simple superalgebra $J(V,f)$ of symmetric bilinear form over an algebraically closed field of characteristic 0 were considered in \cite{MZ}. Here we describe irreducible noncommutative Jordan representations of this superalgebra with restriction that $V_{\bar{0}} \neq 0$ and that $\mathbb{F}$ allows square root extraction (these conditions ensure that the degree of this superalgebra is 2, see below).

\medskip

Let $J = J(V,f)$ be the superalgebra of nondegenerate supersymmetric bilinear form $f$ on a vector superspace $V = V_{\bar{0}} \oplus V_{\bar{1}}$ with $V_{\bar{0}} \neq 0$. Then $J$ has a nonzero idempotent $e = \frac{1}{2} + v,$ where $v \in V_0$ is such that $f(v,v) = \frac{1}{4}.$ The superalgebra $J$ has the following Peirce decomposition relative to $e$:
\[J_0 = \Big\langle \frac{1}{2} - v \Big\rangle,~ J_1 = \{u \in V: f(u,v) = 0\},~ J_2 = \Big\langle \frac{1}{2} + v \Big\rangle.\]

Let $M$ be an irreducible noncommutative Jordan bimodule over $J.$ The one-dimensionality of Peirce spaces $J_0, J_2$ and $1)$ of Lemma \ref{r+r-rel} imply that operators $R_{x}^-, x \in J$, act nonzero only on $M_1$. Now, by $5)$ of Lemma \ref{r+r-comm}, $M' = M_1R_{J}^-$ is a submodule of $M.$ If it is zero, $M$ is commutative (note that in case of algebraically closed field of characteristic 0, finite-dimensional Jordan irreducible representations of $J(V,f)$ were described in \cite{MZ}). If $M' = M,$ then it is easy to see that $R_u^+ = 0$ for any $u \in V$ (for $u \in \langle v \rangle ^{\bot} = J_1$ this follows from Peirce relations, and for $v$ it follows from the fact that $M = M_1$, thus, $R_e^+ = id/2$). One can check that the action of $J$ given by
\[R_1^+ = id,~ R_v^+ = 0, v \in V,\]
is Jordan. Moreover, if one chooses a basis $\{v_i\}$ of $V$, then it is easy to see that relations (\ref{r+r-1}), (\ref{r+r-2}) hold for any set of operators $R_{v_i}^- \in \operatorname{End}(M).$ Hence, we have the following description of irreducible noncommutative Jordan bimodules over $J$:

\begin{Th}
Let $J(V,f)$ be the superalgebra of nondegenerate supersymmetric bilinear form $f$ on vector space $V \neq V_{\bar{1}}$ with a basis $\{v_i\}$ over a field $\mathbb{F}$ which allows square root extraction and $M$ be an irreducible noncommutative Jordan superbimodule over $M.$ Then one of the following holds:\\
$1)$ $M$ is a Jordan superbimodule over $J;$\\
$2)$ $M = M_1, MR_V^+ = 0,$ $R_{v_i}^-$ are linear operators on $M$ such that $M$ has no invariant subspaces with respect to all of $R_{v_i}^-.$
\end{Th}

\medskip

\begin{Remark}One can check that the formulas in the point $2)$ of the previous theorem define a noncommutative Jordan representation not only for $J(V,f)$, but for all algebras $U(V,f,\star).$ Since the superalgebra $D = D_1(\lambda)$ is a superalgebra of the type $U(V,f,\star),$ there are bimodules over it which are not isomorphic to $\operatorname{Reg}(D)$ or $\operatorname{Reg}(D)^{op}.$ Hence, the Theorem \ref{Dt1-rep} and Kronecker factorization theorem do not hold for $t = 1.$
\end{Remark}

\medskip

\subsection{Representations of superalgebras \texorpdfstring{$D_t$ and $K_3$}{Dt and K3}.}
Finite-dimensional unital Jordan representations of $D_t$ and $K_3$ were studied in \cite{MZDt} and \cite{Trush} in the case of characteristic 0, and in \cite{Trushmod} in the case of characteristic $p \neq 2.$ Special representations of $D_t$ were studied in \cite{MZ01}.
In this section we study noncommutative irreducible Jordan bimodules over Jordan superalgebras $D_t$ and $K_3.$

\medskip

First of all, we prove the following useful technical lemma. It will allow us to consider noncommutative Jordan representations of many Jordan superalgebras in a uniform way.
\begin{Lem}
\label{DtM1}
Let $J$ be a unital Jordan superalgebra such that $J$ contains $D_t = \langle e_1, e_2, x, y \rangle, t \neq 1$ as a unital subalgebra $($or $J$ contains $K_3 = \langle e_1, z, w \rangle)$. Then the following statements hold:\\
$1)$ There is no nonzero unital noncommutative Jordan bimodule $M$ over $J$ such that $M = M_1(e_1);$\\
$2)$ If $M$ is a unital noncommutative Jordan bimodule over $J$ such that $(M_0(e_1) + M_2(e_1))R_a^- = 0$ for all $a \in J_0(e_1) + J_2(e_1)$, then $M$ is commutative$;$\\
$3)$ If $J_1(e_1)^2 = J_0(e_1) + J_2(e_1),$ then every noncommutative bimodule over $J$ is commutative.
\end{Lem}
\begin{proof}
$1)$ Suppose first that $J \supseteq D_t, t \neq 1$ as a unital subalgebra and let $M$ be such bimodule. Then $D_t$ also acts unitally on $M.$ Peirce relations (\ref{pd_1}), (\ref{pd_2}) imply that
\[R_x^+ = 0,~ R_y^+ = 0,~ R_{e_1}^+ = R_{e_2}^+ = id/2.\]
From the definition of a noncommutative Jordan representation it follows that the action of $D_t$ on $M$ defined above should be Jordan. However, substituting in (\ref{jord_identity}) $a = x, b = e_1, c = y$ we have $\frac{1}{2} = \frac{1}{2}(\frac{1}{2} + \frac{t}{2}),$ hence, $t = 1,$ which contradicts the lemma condition. If a unital Jordan superalgebra $J$ contains $K_3,$ then it contains its unital hull $D_0$ as a unital subalgebra, and everything follows.

$2)$ From 5) of Lemma \ref{r+r-comm} it follows that $M' = M_1(e_1)R_J^-$ is a $J$-submodule of $M$, and $M' = M_1'$. Hence, the previous point implies that $M_1(e_1)R_J^- = 0.$ From $1)$ of Lemma \ref{r+r-rel} it follows that $(M_0 + M_2)R_x^- = 0$ for all $x \in J_1(e_1),$ therefore, the lemma condition implies that $R_x^- = 0$ for all $x \in J,$ and $M$ is commutative.

$3)$ Follows from $3)$ of Lemma \ref{r+r-comm} and the previous point.
\end{proof}

Now we can describe the representations of the superalgebras $J.$ Let $M$ be a noncommutative bimodule over $J = D_t.$

Suppose first that $t \neq 1.$ Peirce relations and the one-dimensionality of Peirce spaces $J_0, J_2$ imply that $(P_0+P_2)R_a^-=0$ for $a \in J_0 + J_2.$ Then $2)$ of previous lemma implies that every noncommutative Jordan bimodule over $J$ is Jordan.
Consider now the case $t = 1.$ In this case it is easy to see that $D_1$ is a Jordan superalgebra of nondegenerate symmetric form on the space $V = \langle e_1 - e_2, x, y \rangle$ with $V_{\bar{0}} = \langle e_1 - e_2 \rangle, V_{\bar{1}} = \langle x, y \rangle ,$ and irreducible bimodules over superalgebras of superforms were classified in the previous subsection. Hence, we have proved the following results:

\begin{Th}
\label{Dt-rep}
Every noncommutative Jordan bimodule over $D_t, t \neq 1,$ is Jordan.
\end{Th}

\begin{Th}
\label{D1-rep}
Let $M$ be an irreducible noncommutative Jordan bimodule over $D_1$. Then one of the following holds:\\
$1)$ $M$ is a Jordan bimodule;\\
$2)$ $M = M_1(e_1), MR_x^+ = MR_y^+ = 0,$ $R_x^-, R_y^-, R_{e_1}^- = -R_{e_2}^-$ are linear operators on $M$ such that $M$ has no invariant subspaces with respect to all of them.
\end{Th}

From 3) of Lemma \ref{DtM1} immediately follows the following theorem:
\begin{Th}
\label{K3-rep}
Every noncommutative Jordan bimodule over $K_3$ is Jordan.
\end{Th}

\subsection{Representations of the superalgebra \texorpdfstring{$P(2)$}{P(2)}.}
Recall that the simple Jordan superalgebra $P(n) \cong \operatorname{H}(M_{n,n}(\mathbb{F}),\operatorname{strp})$ is the Jordan superalgebra of symmetric elements of the simple associative superalgebra $M_{n,n}(\mathbb{F})$ with respect to the transpose superinvolution
\[\begin{pmatrix} A & B \\ C & D \end{pmatrix}^{strp} = \begin{pmatrix} D^t & -B^t \\ C^t & A^t \end{pmatrix},\]
where $A,B,C,D \in M_n(\mathbb{F})$, and $t$ is the transpose. Jordan representations of $P(n)$ were described in \cite{MZ} in the case of algebraically closed field of characteristic 0 and $n \geq 2,$ and in \cite{MSZ} in the case of arbitrary field and $n \geq 3.$

\medskip

In the paper \cite{ps2} it was proved that $P(2)$ does not admit a nonzero generic Poisson bracket. The degree of $P(n)$ is exactly $n$, so here we will only deal with noncommutative Jordan representations of $P(2)$ (the superalgebra $P(1)$ is not simple).
Representations of $P(2)$ are characterized in the following theorem:

\begin{Th}
\label{P2-rep}
All noncommutative Jordan representations of $P(2)$ are Jordan.
\end{Th}

\begin{proof} Let $e_1 = e_{11} +e_{33}$ be an idempotent of $P(2) = J.$ We have the following Peirce decomposition relative to $e_1$:
\begin{gather*}
U_0 = \langle e_2:=e_{22}+e_{44}, f:=e_{42} \rangle,\\
U_1 = \langle a:= e_{12} + e_{43}, b:= e_{21} +e_{34}, c:= e_{14} -e_{23}, d:=e_{32}+e_{41} \rangle,\\
U_2 = \langle e_1:= e_{11} + e_{33}, e:=e_{31} \rangle.
\end{gather*}
For the sake of convenience we provide the multiplication table of $P(2)$ (zero products are omitted):
\begin{table}[h]
\centering
\label{my-label}
\begin{tabular}{|l|l|l|l|l|l|l|l|l|}
\hline
$\circ$ & $e_1$ & $e_2$ & $e$    & $f$   & $a$           & $b$           & $c$           & $d$           \\ \hline
$e_1$   & $e_1$ &       & $e$    &       & $a/2$         & $b/2$         & $c/2$         & $d/2$         \\ \hline
$e_2$   &       & $e_2$ &        & $f$   & $a/2$         & $b/2$         & $c/2$         & $d/2$         \\ \hline
$e$     & $e$   &       &        &       & $d/2$         &               & $b/2$         &               \\ \hline
$f$     &       & $f$   &        &       &               & $d/2$         & $-a/2$        &               \\ \hline
$a$     & $a/2$ & $a/2$ & $d/2$  &       &               & $(e_1+e_2)/2$ &               & $f$           \\ \hline
$b$     & $b/2$ & $b/2$ &        & $d/2$ & $(e_1+e_2)/2$ &               &               & $e$           \\ \hline
$c$     & $c/2$ & $c/2$ & $-b/2$ & $a/2$ &               &               &               & $(e_1-e_2)/2$ \\ \hline
$d$     & $d/2$ & $d/2$ &        &       & $f$           & $e$           & $(e_2-e_1)/2$ &               \\ \hline
\end{tabular}
\end{table}

Note that $J$ has a (unital) subalgebra $J' = \langle e_1, e_2, c, d \rangle$ which is isomorphic to $D_{-1}.$ Also it is easy to see that $J_1(e_1) \circ J_1(e_1) = J_0(e_1) + J_2(e_1).$ Thus, $3)$ of Lemma \ref{DtM1} implies that every noncommutative Jordan bimodule over $J$ is Jordan.

\end{proof}

\subsection{Representations of the superalgebra \texorpdfstring{$Q(2)^{(+)}$}{Q(2)(+)}}

Finite-dimensional Jordan represenations of a simple Jordan superalgebra $Q(2)^{(+)}$ were studied in \cite{MSZ}. In particular, irreducible bimodules were described and it was proved that if the characteristic of the field is zero or $>3$, then every finite-dimensional representation over $Q(2)^{(+)}$ is completely reducible.

\medskip

Here we describe noncommutative Jordan bimodules over $Q(2)^{(+)}.$

Recall that $Q(n) = M_n(\mathbb{F}) \oplus \overline{M_n(\mathbb{F})},$ where $\overline{M_n(\mathbb{F})}$ is an isomorphic copy of $M_n(\mathbb{F})$ as a vector space. Also, $Q(n)_{\bar{0}} = M_n(\mathbb{F}), Q(n)_{\bar{1}} = \overline{M_n(\mathbb{F})}.$ The multiplication in $Q(n)$ is defined like this:
\[a\cdot b = ab,~ \bar{a} \cdot b = a \cdot \bar{b} = \overline{ab},~ \bar{a} \cdot \bar{b} = ab,\]
Therefore, the multiplication is defined in $Q(n)^{(+)}$ in the following way:
\[a \circ b = a \circ b,~ a \circ \bar{b} = \bar{a} \circ b = \overline{a \circ b},~ \bar{a} \circ \bar{b} = [a,b]/2,\]
where $a, b \in M_n(\mathbb{F})$, and $a \circ b$, $[a,b]$ are the matrix Jordan product and commutator.

Regarding to the idempotent $e_{11}$, we have the following Peirce decomposition of $Q(2)^{(+)} = J:$
\[J_0 = \langle e_{22}, \overline{e_{22}} \rangle,~J_1 = \langle e_{12}, e_{21}, \overline{e_{12}}, \overline{e_{21}} \rangle,~J_2 = \langle e_{11}, \overline{e_{11}}\rangle.\]

Note that $J$ has a unital subsuperalgebra $J' = \langle e_{11}, e_{22}, \overline{e_{12}}, \overline{e_{21}} \rangle$, which is isomorphic to $D_{-1}.$

\medskip

Now we are ready to study the representations of $J.$ Let $M$ be a unital noncommutative Jordan bimodule over $J,$ and let $M = M_0 + M_1 + M_2$ be its Peirce decomposition with respect to $e_{11}.$
Substituting $a = e_{12}, b = \overline{e_{21}}$ in (\ref{r+r-2}), by 4) of Lemma \ref{r+r-comm} we have
$0 = R_{\overline{e_{11}}}^- + R_{\overline{e_{22}}}^-.$
Therefore, Peirce relations imply that
\[P_0R_{\overline{e_{22}}}^- = -P_0R_{\overline{e_{11}}}^- = 0.\]
Analogously, $P_2R_{\overline{e_{11}}}^- = 0.$ Hence,
\[(P_0 + P_2)R_{\overline{e_{11}}}^- = (P_0 + P_2)R_{\overline{e_{22}}}^- = 0.\]
Combining this with Peirce relations, we have $(M_0 + M_2)R_{J_0 + J_2}^- = 0.$ Thus, 2) of lemma \ref{DtM1} implies that $M$ is Jordan. We state this result as a theorem:
\begin{Th}
Every noncommutative Jordan bimodule over $Q(2)^{(+)}$ is Jordan.
\end{Th}

\subsection{Representations of superalgebras \texorpdfstring{$K_{10}$ and $K_9$}{K10 and K9}.}

Representations of the Kac superalgebra $K_{10}$ were studied in the case of algebraically closed field of characteristic 0. In the article \cite{Sht} Shtern proved that any (Jordan) representation of $K_{10}$ is completely reducible, with irreducible summands being the regular module and its opposite. Later, Mart\'{i}nez and Zelmanov used his results to prove the Kronecker factorization theorem for $K_{10}.$

\medskip

The superalgebras $K_{10}$ and $K_9$ do not admit nonzero generic Poisson brackets (see \cite{ps3}). In this section we classify noncommutative Jordan representations over $K_{10}$ and $K_9.$ Also we prove the Kronecker factorization theorem for $K_{10}$ in the case where the base field is algebraically closed and of characteristic 0.

\medskip

Recall the definitions of the simple Jordan superalgebras $K_{10}$ and $K_{9}$ over a field $\mathbb{F}$.
The even and odd parts of $K_{10}$ are respectively
\[A=A_1\oplus A_2= \left<e_1,uz,vz,uw,vw\right>
 \oplus \left<e_2\right> \text{ and } M=\left<u,v,w,z\right>\]
The even part $A$ is a direct sum of ideals (of $A$).
 The unity in $A_1$ is $e_1$, and $e_i\cdot m=\frac{1}{2}m$ for every $m\in M$.
 The multiplication table of $K_{10}$ is the following:
\begin{gather*}
u\cdot z=uz,u\cdot w=uw,v\cdot z=vz,v\cdot w=vw,\\
z\cdot w=e_1-3e_2,uz\cdot w=-u,vz\cdot w=-v,uz\cdot vw=2e_1,
\end{gather*}and the remaining nonzero products may be obtained either by applying the skew-symmetries $z\leftrightarrow w$, $u\leftrightarrow v$, or by the substitution $z\leftrightarrow u$, $w\leftrightarrow v$. If the characteristic of $\mathbb{F}$ is not $3$, the superalgebra $K_{10}$ is simple; but in the case of characteristic $3$ it contains a simple subsuperalgebra $K_{9}=A_1\oplus M$.

Consider first the case of superalgebra $J = K_{10}.$ One can see that it contains a subsuperalgebra $\langle e_1, e_2, z, w \rangle,$ which is isomorphic to $D_{-3}.$ The Peirce decomposition of $J$ with respect to $e_1$ is the following:
\[J_0(e_1) = A_2,~ J_1(e_1) = M,~ J_2(e_1) = A_1.\]

Let $M$ be a noncommutative Jordan bimodule over $K_{10}.$ From the multiplication table it is easy to see that
\[uz,vz,uw,vw \in J_1(e_1)^2,\]
thus, 3) of Lemma \ref{r+r-comm} implies that
\[R_{uz}^- = 0,~ R_{vz}^- = 0,~ R_{uw}^- = 0,~ R_{vw}^- = 0 \text{ on } M.\]
Also, since $e_1$ and $e_2$ are orthogonal idempotents, it is obvious that $(M_0(e_1) + M_2(e_1))R_{e_1}^- = (M_0(e_1) + M_2(e_1))R_{e_2}^- = 0.$ Hence, 2) of Lemma \ref{DtM1} implies that $M$ is Jordan.

\medskip

Now consider the case of $K_9 = J.$ One can see that it contains a subsuperalgebra $\langle e_1, z, w \rangle \cong K_3.$ The Peirce decomposition of $J$ with respect to $e_1$ is the following:
\[J_0(e_1) = 0,~ J_1(e_1) = M,~ J_2(e_1) = A_1.\]

Let $M$ be a noncommutative Jordan bimodule over $K_{9}.$ Again, 3) of Lemma \ref{r+r-comm} implies that
\[R_{uz}^- = 0,~ R_{vz}^- = 0,~ R_{uw}^- = 0,~ R_{vw}^- = 0 \text{ on } M.\]
Also, it is obvious that $(M_0(e_1) + M_2(e_1))R_{e_1}^- = 0.$ Hence, 2) of Lemma \ref{DtM1} implies that $M$ is Jordan. We have proved the following theorem:

\begin{Th}
Every unital noncommutative Jordan bimodule over $K_{10}$ or $K_9$ is Jordan.
\end{Th}

\medskip

In conclusion, we prove the Kronecker factorization theorem for the superalgebra $K_{10}$ in the case where $\mathbb{F}$ is algebraically closed of characteristic 0. Mart\'{\i}nez and Zelmanov proved that a Jordan superalgebra containing $K_{10}$ as a unital subsuperalgebra is isomorphic to $Z \otimes K_{10}$ for an associative-supercommutative superalgebra $Z.$ We extend their result to noncommutative Jordan case:

\begin{Th}
Suppose that the base field $\mathbb{F}$ is algebraically closed and is of characteristic $0$. Let $U$ be a noncommutative Jordan superalgebra which contains $J = K_{10}$ as a unital subsuperalgebra. Then $U$ is supercommutative and $U \cong Z \otimes J,$ where $Z$ is an associative-supercommutative superalgebra.
\end{Th}
\begin{proof} Let $U = U_0 + U_1 + U_2$ be the Peirce decomposition of $U$ with respect to $e_1.$ We need to show that $U$ is supercommutative. Since $U$ is a Jordan bimodule over $J,$ the point 1) of lemma \ref{r+r-rel} implies that $[U_0 + U_2, U_1] = 0, [U_1,U_1] \subseteq U_1.$

We prove that we can take $K = \{u,z\}$ in the lemma \ref{K_set}. Since $U$ is a Jordan superbimodule over $J$, the Peirce relations (\ref{pd_2}) imply that $KU_1 \subseteq U_0 + U_2.$ Suppose now that $K \circ a = 0$ for $a \in U_1.$ The description of Jordan superbimodules over $J$ (\cite{Sht}) implies that $U$ is a direct sum of submodules isomorphic either to $\operatorname{Reg}(K_{10})$ or $\operatorname{Reg}(K_{10})^{op}.$ Hence, we can assume that $a \in \operatorname{Reg}(J)$ or $\operatorname{Reg}(J)^{op}.$ Let
\[a = \alpha u' + \beta v' + \gamma w' + \delta z' \in \operatorname{Reg}(J)\]
(we have added dashes to distinguish elements of the regular superbimodule from elements of $J$). Then
\[u \circ a = \beta(e_1' -3e_2') + \gamma(uw)' + \delta(uz)' = 0,\]
thus, $\beta = \gamma = \delta = 0.$ Hence, $z \circ a = -\alpha(uz)' = 0,$ and $\alpha = 0.$
Therefore, $a = 0.$ The case where $a \in \operatorname{Reg}(J)^{op}$ is considered analogously. We proved that $K$ satisfies the conditions of Lemma \ref{K_set}, thus, $U_1^2 \subseteq U_0 + U_2.$ Hence, $[U_1,U_1] = 0$ and $[U,U_1] = UR_{U_1}^- = 0.$

The structure of $U$ as a bimodule over $K_{10}$ implies that
\[U_0 + U_2 = (J_1 \circ U_1) + (J_1 \circ U_1) \circ (J_1 \circ U_1).\]
 The relation (\ref{r+r-2}) implies that $R_{J_1 \circ U_1}^- = 0,$ and, applied again, shows that $R_{(J_1 \circ U_1) \circ (J_1 \circ U_1)}^- = 0.$ Therefore, $R_{U_0 + U_2}^- = 0$ and $U$ is commutative. By the result of Mart\'{\i}nez and Zelmanov, $U \cong Z \otimes J$ for an associative-supercommutative superalgebra $Z.$

\end{proof}
In fact we can drop the assumption that $K_{10}$ contains the unity of $U.$

\begin{Th}
Let $U$ be a noncommutative Jordan superalgebra that contains $J = K_{10}$ as a subsuperalgebra. Then $U \cong (Z \otimes J) \oplus U',$ where $Z$ is an associative-supercommutative superalgebra.
\end{Th}
\begin{proof} Let $U = U_0 + U_1 + U_2$ be the Peirce decomposition of $U$ with respect to the unity of $J.$ Then $U^{(+)}$ is a Jordan superalgebra and $U_1$ is a special Jordan bimodule over $J$ with the action induced by multiplication in $U.$ But since $J$ is an exceptional simple Jordan superalgebra, $U_1$ must be zero (see, for example, \cite{MZKr}). Hence, $U = U_0 + U_2.$ Applying the previous theorem to $J \subseteq U_2,$ we get the desired result.
\end{proof}

\begin{Remark}In this article we have considered representations of all simple noncommutative Jordan superalgebras of degree $\geq 2$ (in case where char $\mathbb{F} = 0$) except superalgebras $U(V,f,\star)$ where $\star \neq 0$ and $K(\Gamma_n,A)$ (see theorem \ref{classification}). Representations of these superalgebras will be considered in the upcoming papers.
\end{Remark}

\bigskip

{\bf Acknowledgements.}

The author is grateful to Prof. Alexandre Pozhidaev (Sobolev Institute of Math., Russia) and Prof. Ivan Kaygorodov (UFABC, Brazil)
for interest and constructive comments.



\end{document}